\documentclass[12pt]{amsart}
\usepackage{amsmath,amsthm}
\usepackage{amssymb}
\usepackage[all]{xypic}
\SelectTips{cm}{12}
\UseTips{}
\usepackage{euscript}
\theoremstyle{plain}
\swapnumbers
    \newtheorem{thm}{Theorem}[section]
    \newtheorem*{thma}{Theorem A}
    \newtheorem*{thmb}{Theorem B}
    \newtheorem*{thmc}{Theorem C}
    \newtheorem{prop}[thm]{Proposition}
    \newtheorem{lemma}[thm]{Lemma}
    \newtheorem{cor}[thm]{Corollary}
    \newtheorem{fact}[thm]{Fact}
    \newtheorem{subsec}[thm]{}

\theoremstyle{definition}
    \newtheorem{defn}[thm]{Definition}
    \newtheorem{example}[thm]{Example}
    
    \newtheorem{notation}[thm]{Notation}
\theoremstyle{remark}
        \newtheorem{remark}[thm]{Remark}
        
	\newtheorem{ack}[thm]{Acknowledgements}
\newenvironment{myeq}[1][]
{\stepcounter{thm}\begin{equation}\tag{\thethm}{#1}}
{\end{equation}}

\newcommand{\mydiagram}[2][]
{\stepcounter{thm}\begin{equation}
     \tag{\thethm}{#1}\vcenter{\xymatrix{#2}}\end{equation}}
\newenvironment{mysubsection}[2][]
{\begin{subsec}\begin{upshape}\begin{bfseries}{#2.}
\end{bfseries}{#1}}
{\end{upshape}\end{subsec}}
\newenvironment{mysubsect}[2][]
{\vsn\quad \begin{subsec}\begin{upshape}\begin{bfseries}{{#2}\vsn.}
\end{bfseries}{#1}}
{\end{upshape}\end{subsec}}
\newcommand{\w}[2][ ]{\ \ensuremath{#2}{#1}\ }
\newcommand{\wref}[2][ ]{\ \eqref{#2}{#1}\ }
\newcommand{\xra}[1]{\xrightarrow{#1}}

\newcommand{\hra}{\hookrightarrow}

\newcommand{\hsm}{\hspace{2 mm}}

\newcommand{\vsn}{\vspace{0.5 mm}}
\newcommand{\vsm}{\vspace{3 mm}}
\newcommand{\rest}[1]{\lvert_{#1}}

\newcommand{\equaliz}{~\raisebox{-0.6ex}{$\stackrel{\textstyle 
            \longrightarrow}{\longrightarrow}$}~}
\newcommand{\good}[1][ ]{strongly directed{#1}}
\newcommand{\Arr}{\operatorname{Arr}}
\newcommand{\Cof}{\operatorname{Cof}}
\newcommand{\colim}{\operatorname{colim}}
\newcommand{\colimit}[1]{\raisebox{-0.25cm}{$\stackrel{\textstyle\colim}
{\mbox{\scriptsize{${#1}$}}}$}\,\vspace{0.4cm}}

\DeclareMathOperator{\Diag}{Diag}
\newcommand{\diag}[3]{\Diag_{#1}\left({#2},{#3}\right)}
\newcommand{\diagk}[4]{\Diag_{#1}^{#4}\left({#2},{#3}\right)}
\newcommand{\Fib}{\operatorname{Fib}}
\DeclareMathOperator{\holim}{holim}
\DeclareMathOperator{\hocolim}{hocolim}
\newcommand{\ho}{\operatorname{ho}}
\newcommand{\Hom}{\operatorname{Hom}}
\newcommand{\Id}{\operatorname{Id}}

\newcommand{\Ker}{\operatorname{Ker}}
\newcommand{\lof}[3]{L_{#1}\left({#2},{#3}\right)}
\newcommand{\hlof}[3]{\cL_{#1}\left({#2},{#3}\right)}
\DeclareMathOperator{\Map}{map}
\newcommand{\map}{\Map\,}

\DeclareMathOperator{\Nat}{Nat}
\newcommand{\nat}[3]{\Nat_{#1}({#2},{#3})}
\DeclareMathOperator{\Indec}{Indec}
\newcommand{\indec}[2][]{\Indec_{#1}({#2})}
\newcommand{\Obj}{\operatorname{Obj}}
\newcommand{\op}{\operatorname{op}}

\newcommand{\proj}{\operatorname{proj}}

\newcommand{\Trans}{\cD}
\newcommand{\trans}[3]{\Trans_{#1}({#2},{#3})}
\newcommand{\ctrans}[4]{\cE_{#1}^{#2}({#3},{#4})}
\newcommand{\Phiof}[2]{\Phi^{#1}_{#2}}
\newcommand{\Phis}[1]{\Phiof{}{#1}}
\newcommand{\psiof}[1]{\psi \left( #1 \right)}

\newcommand{\nuof}[1]{\nu\left({#1}\right)}
\newcommand{\sigof}[1]{\,\sigma\left({#1}\right)}
\newcommand{\phiof}[1]{\varphi_{#1}}

\newcommand{\mdiag}[1]{M(#1)}
\newcommand{\mf}{\mdiag{f}}
\newcommand{\mfd}{\mdiag{\fdot}}
\newcommand{\comp}[1]{c({#1})}

\newcommand{\fdot}{f_{\bullet}}
\newcommand{\gdot}{g_{\bullet}}
\newcommand{\hdot}{h_{\bullet}}
\newcommand{\kdot}{k_{\bullet}}
\newcommand{\ldot}{\ell_{\bullet}}
\newcommand{\phidot}{\varphi_{\bullet}}
\newcommand{\sigdot}{\sigma_{\bullet}}

\newcommand{\cA}{{\mathcal A}}

\newcommand{\cC}{{\mathcal C}}
\newcommand{\CI}{\cC^{I}}
\newcommand{\CIn}[1]{\cC^{#1}}
\newcommand{\CIZ}{\CI/Z}
\newcommand{\CInZ}[1]{\CIn{#1}/Z}
\newcommand{\CJ}{\cC^{J}}
\newcommand{\tuple}{(\cC,I,Z,X,Y)}
\renewcommand{\cD}{{\mathcal D}}
\newcommand{\cE}{{\mathcal E}}
\newcommand{\cF}{{\mathcal F}}
\newcommand{\cG}{{\mathcal G}}
\renewcommand{\cH}{{\mathcal H}}
\newcommand{\cJ}{{\mathcal J}}
\renewcommand{\cL}{{\mathcal L}}
\newcommand{\cM}{{\mathcal M}}
\newcommand{\cN}{{\mathcal N}}
\newcommand{\cO}{{\mathcal O}}

\newcommand{\fG}{\mathfrak{G}}
\newcommand{\Alg}[1]{{#1}\text{-}{\EuScript Alg}}
\newcommand{\Abgp}{{\EuScript AbGp}}
\newcommand{\Gp}{{\EuScript Gp}}
\newcommand{\Pa}{$\Pi$-algebra}
\newcommand{\PAlg}{\Alg{\Pi}}
\newcommand{\Top}{{\EuScript Top}}
\newcommand{\XX}{{\EuScript X}}
\newcommand{\hJ}[1]{\widetilde{J}_{#1}}
\newcommand{\fJ}[1]{J[{#1}]}
\newcommand{\hoa}{\hat{\omega}}
\newcommand{\hsi}{\hat{\sigma}}

\newcommand{\hX}{\hat{X}}

\newcommand{\bY}{\bar{Y}}
\newcommand{\bz}{\mathbf{0}}
\newcommand{\bo}{\mathbf{1}}

\newcommand{\bk}{\mathbf{k}}
\newcommand{\bn}{\mathbf{n}}
\newcommand{\bnm}{\mathbf{n-1}}

\newcommand{\ZZ}{\mathbb Z}

\begin{document}

\title[Spectral sequences for the cohomology of diagrams]
{Local-to-global spectral sequences for the cohomology of diagrams}
 \author[D.~Blanc]{David Blanc}
 \address{Department of Mathematics\\
    University of Haifa\\ 31905 Haifa, Israel}
 \email{blanc@math.haifa.ac.il}
\author [M.W.~Johnson]{Mark W.~Johnson}
\address{Department of Mathematics\\
         Penn State Altoona\\ Altoona, PA 16601, USA}
\email{mwj3@psu.edu}
\author[J.M.~Turner]{James M.~Turner}
\address{Department of Mathematics\\
        Calvin College\\ Grand Rapids, MI 49546, USA}
\email{jturner@calvin.edu}
\date{July 9, 2007; revised April 3, 2008}
\subjclass{Primary: 55N99; \ secondary: 18G55,18G40,18G55}
\keywords{Diagrams, Andr\'{e}-Quillen cohomology, spectral sequences, 
local-to-global}

\maketitle

%
%
\section{Introduction}\label{sint}

The cohomology of diagrams arises as a natural object of study in
several mathematical contexts: 
in deformation theory (see \cite{GSchaA,GSchaD,GGSchaD}),
and in classifying diagrams of groups, as in  \cite{CegaCD}. 
If $I$ is the one-object category  corresponding to a
group $G$, a diagram \w{X\in\CI} is just an object in $\cC$ 
equipped with a $G$-action, and its cohomology is the equivariant 
cohomology of \cite{IllmE1} (cf.\ \cite[\S 2]{PiacCD}). On the other
hand, for any discrete or Lie group $G$, let \w{I=\cO_{G}} denote the
orbit category of $G$: if $X$ is a $G$-space, \w{\XX:\cO_{G}\to\Top}
is the corresponding fixed point diagram \w[,]{\XX(G/H):=X^{H}} and
\w[,]{M:\cO_{G}\to\Abgp} is any system of coefficients, then the
corresponding cohomology \w{H(\XX;\cM)} is Bredon cohomology (cf.\
\cite[I,\S 4]{MayEHC}). 
Finally, when $I$ consists of a single arrow, and the
coefficients are constant, we have the usual cohomology of a pair. 
   See \cite{BGuinCD}, \cite{DSchuD}, \cite{FWillB}, \cite{OlumHS},
\cite{PaveG}, and \cite{BCarlD} for further applications. 

\begin{mysubsection}{Diagrams in homotopy theory}
\label{sdht}
The cohomology of diagrams also plays a major role in the
Dwyer-Kan-Smith theory for the rectification of homotopy-commuta\-tive
diagrams (cf.\ \cite{DKSmH} and \cite{DroHH,DKaHM}).
In fact, our interest in the subject was motivated by the related
realization problem for diagrams of \Pa s (graded groups with an action
of the primary homotopy operations): as in the case of a single \Pa\ 
(cf.\ \cite{BDGoerR}), the obstructions to realizing a diagram of \Pa s\
\w{\Lambda:I\to\PAlg} lie in appropriate cohomology groups of $\Lambda$
(see \cite[Thm.\ 6.3]{BJTurR}).

Furthermore, given a \Pa\ $\Gamma$, all distinct homotopy types 
realizing $\Gamma$ may be distinguished by a set of higher homotopy 
operations associated to a collection \w{(I^{\alpha})_{\alpha\in A}} of
finite indexing categories \w{I^{\alpha}} and homotopy-commutative
diagrams \w[,]{X^{\alpha}:I^{\alpha}\to \ho\Top} where all the spaces 
\w{X_{i}^{\alpha}} are wedges of spheres. Since these higher
operations are obstructions to the rectification of the diagrams
\w{(X^{\alpha})_{\alpha\in A}} (and thus the associated diagrams 
\w[),]{\Lambda^{\alpha}:=\pi_{\ast}X^{\alpha}:I^{\alpha}\to\PAlg} \ they
correspond to elements in the cohomology of $\Gamma$. Understanding the
cohomology groups of such diagrams may therefore be helpful in algebraicizing
(and organizing) the ``higher \Pa'' of a space $Y$, consisting of all higher
homotopy operations in \w[.]{\pi_{\ast} Y}
\end{mysubsection}

\begin{mysubsection}{Computing diagram cohomology}
\label{scdc}
Even the cohomology of a single map may be hard to calculate
(cf.\ \cite[\S 5.16]{BJTurR}), so some computational tools are needed.
For this purpose we construct ``local-to-global'' spectral sequences
for the cohomology of a diagram, which can be used to compute the cohomology
of the full diagram in terms of smaller pieces.

Given a small category $I$, a model category $\cC$ (in
the sense of \cite{QuiH}), and an $I$-diagram \w[,]{X\in\CI} one can
define the cohomology of $X$ with coefficients in any abelian group
object \w[.]{Y\in\CI} For technical reasons, we shall concentrate on the case
where \w{\cC=s\cA} is the category of simplicial objects over
some variety of universal algebras $\cA$: since the homotopy category
of simplicial groups is equivalent to that of (pointed connected)
topological spaces, this actually covers all cases of interest above.
Some of our results are valid, however, for an arbitrary simplicial
model category $\cC$.

Another reason for our interest in the ``local-to-global'' approach to
diagram cohomology is that in order for the higher homotopy
operation corresponding to a homotopy commutative diagram 
\w{X:I\to\ho\Top} to be \emph{defined}, all lower order operations 
(corresponding to subdiagrams of $I$) must vanish \emph{coherently}.
Thus an essential step in a cohomological description of higher order
operations is the ability to piece together local data to obtain
global information. 
\end{mysubsection}

\begin{remark}\label{rlg}
We should point out that our methods work (almost exclusively) for a
\emph{directed} indexing category $I$ (i.e., with only identities as
endomorphisms), which is a significant restriction.  However, such
diagrams certainly suffice for the description of higher homotopy
operations, as above: even the linear case \ -- \  when $I$ consists
of a single composable sequence of arrows \ -- \ is of interest, since
the realizability of such a diagram is essentially equivalent to
calculating higher Toda brackets. Furthermore, diagrams arising in
deformation theory (indexed by the nerve of a covering) are of this
form. Our methods, suitably modified (cf.\ Remark \ref{rgroup}), also
apply to diagrams indexed by the orbit category \w{\cO_{G}} of a group $G$.
\end{remark} 

\begin{mysubsection}{The spectral sequences}
Let $\cC$  be a simplicial model category and $I$ a directed index category,
and assume given diagrams \w[,]{Z:I\to\cC} and \w[,]{X,Y\in\CIZ} 
with $Y$ an abelian group object in \w[.]{\CIZ} Our main results may 
be summarized as follows:

\begin{thma}
There is a first quadrant spectral sequence with:
\[
E^{2}_{s,t}~=~\prod_{j\in\hJ{s}}~H^{t+s}(X_{j}/Z_{j},\hat{\phi}_{j})
~\Longrightarrow H^{s+t}(X/Z;\,Y)
\]
\end{thma}

This is constructed by taking increasing truncations of the coefficient
diagram $Y$ (cf.\ Theorem \ref{tone}). Here \w{H^{\ast}(X/Z,\phi)} 
denotes relative cohomology for a map of the coefficients 
(see Definition \ref{drelc}).

\begin{thmb}
There is a first quadrant spectral sequence with:
\[
E^{2}_{s,t}~=~H^{s+t}(\eta_{s};\,Y)~\Longrightarrow H^{s+t}(X/Z;\,Y)
\]
\end{thmb}

This spectral sequence is constructed dually to the previous one, by taking 
increasing truncations of the \emph{source} diagram $X$ 
(see Theorem \ref{ttwo}). Here \w{H^{\ast}(\eta,Y)} denotes the usual 
cohomology of a map (or pair).

\begin{thmc}
If $I$ is countable, then for any ordering  \w{(c_{s})_{s=1}^{\infty}}
of the objects of $I$, there is a first quadrant spectral sequence with
\w[.]{E^{2}_{s,t}~=~H^{t+s}_{c_{s}}(X/Z;Y)~\Longrightarrow H^{s+t}(X/Z;\,Y)}
\end{thmc}

This is constructed by successively omitting the objects \w{c_{s}} 
from $I$ (see Theorem \ref{tthree}).
Here \w{H^{\ast}_{c}(X/Z,Y)} denote the local cohomology groups 
at an object \w{c\in I} (see Definition \ref{dlocoh}).

There are versions of all three spectral sequences defined for any suitable
cover $\cJ$ of $I$ (Definition \ref{dcover}). In particular, the 
spectral sequences always converge if $\cJ$ is finite, hence if $I$ itself
is finite.
\end{mysubsection}

\begin{mysubsection}{Other variants}
\label{sov}
Other spectral sequences for the cohomology of a diagram  have appeared in 
the literature. One should mention the universal coefficients
spectral sequence of Piacenza (see \cite[\S 1]{PiacD}), the the
$p$-chain spectral sequence of Davis and L\"{u}ck (see \cite{DLuckP}), 
the equivariant Serre spectral sequence of Moerdijk and Svensson 
(see \cite{MSvenS}, and the local-to-global spectral sequences of
Jibladze and Pirashvili (see \cite{JPirCA}) and Robinson (see
\cite{MRobiC}) \ -- \ though the last three use a different definition
of cohomology, based on the Baues-Wirsching and Hochschild-Mitchell
cohomologies of categories (cf.~\cite{BWirC,BMitR}). 
\end{mysubsection}

\begin{mysubsection}{Organization}
\label{sorg}
Section \ref{ccd} provides background material on diagrams, their covers, 
and the model category of diagrams. In Section \ref{ctow} we determine 
when the ``restriction tower'' associated to a cover of the indexing category
$I$ is a tower of fibrations, and in Section \ref{css} we use this to set up
the first two spectral sequences.

The second half of the paper is devoted to the (somewhat more technical) 
approach based on ``localizing at an object'': 
Section \ref{calc} provides the setting, and explains the method. 
In Section \ref{cauxt} we describe an auxiliary construction associated to
the tower of certain covers of $I$, and in Section \ref{cfibat} show
that this auxiliary tower is a tower of fibrations. Finally, in 
Section \ref{cif} we identify the fibers of the new tower, and obtain the
third spectral sequence.
\end{mysubsection}

\begin{ack}
We would like to thank the referee for his or her comments.
This research was supported by BSF grant 2006039; the third author was 
also supported by NSF grant DMS-0206647 and a Calvin Research Fellowship (SDG).
\end{ack}

%
%
\section{The category of diagrams}
\label{ccd}

Our object of study will be the category \w{\CI} of diagrams \ -- \ i.e.,
functors from a fixed small (often finite) indexing category $I$ 
into a model category $\cC$. The maps are natural transformations. 
In this section we define some concepts and introduce notation related 
to $I$ and \w[:]{\CI}

\begin{defn}\label{dcover}
Let $I$ be any small category. By an $N$-indexed \emph{cover} of $I$ we
mean some collection \w{\cJ=\{J_{\nu}\}_{\nu\in N}} of subcategories
of $I$, such that each arrow in $I$ belongs to at least one \w[.]{J_{\nu}}

A cover \w{\cJ=\{J_{\nu}\}_{\nu\in N}} for $I$ will be
called \emph{orderable} if the relation:
\[
\nu_{1}\prec \nu_{2}~~\stackrel{\text{Def}}{\Longleftrightarrow}~~
\exists~i_{1}\in J_{\nu_{1}},~i_{2}\in J_{\nu_{2}}~
\exists~\phi:i_{2}\to i_{1}~\text{in}~ I \ \text{with}~
i_{1}\not\in J_{\nu_{2}}\ \text{or}~i_{2}\not\in J_{\nu_{1}}~.
\]
defines a partial order on $N$, and the partially ordered set \w{(N,\prec)}
can be embedded as a (possibly infinite) segment of \w[.]{(\ZZ,\leq)}
Choosing such an embedding \w[,]{N\subseteq\ZZ} we may think of $\cJ$ as
indexed by integers, and we can then filter $I$ by setting 
\w[.]{\fJ{n}:=\bigcup_{i\leq n}~J_{i}} 
If $N$ is bounded below in $\ZZ$ we say that $\cJ$ is 
\emph{right-orderable}, and if it is bounded above we say it is
\emph{left-orderable}. 
\end{defn}

\begin{remark}\label{rorder}
Note that the linear ordering of $\cJ$ (indicated by the indices)
is not generally uniquely determined by the partial order $\prec$:
there may be elements of $\cJ$ which are not comparable under $\prec$.
This happens when all maps out of \w{J_{n}} actually land in \w{\fJ{k}} 
for \w[.]{k<n-1} In this case the linear ordering of 
\w{J_{n}} and \w[,]{J_{n-1}} for example, may be switched with impunity.
\end{remark}

\begin{mysubsection}{Directed indexing categories}
\label{sdirect}
A \emph{directed indexing category} is a small category $I$ equipped 
with a map \w[,]{\deg:\Obj(I)\to\ZZ} such that for every non-identity 
map \w{\phi:j\to i} in $I$, \w[.]{\deg(j)>\deg(i)} 
Then $I$ is filtered by the full subcategories \w{I_{n}=\fJ{n}} whose objects 
have degree at most $n$.

An orderable cover \w{\cJ=\{J_{n}\}_{n\in N}} for such an $I$ will
be called \emph{compatible} (with the choice of \w[)]{\deg} if there
is a strictly increasing sequence of integers \w{(k_{n})_{n\in N}}
such that \w[.]{\Obj(J_{n})=\deg^{-1}([k_{n-1},k_{n}])}
\end{mysubsection}

\begin{example}\label{emainex}
The \emph{fine cover} for a directed indexing category $I$ is defined
by letting \w{J_{n}} be the subcategory obtained from the
``difference categories'' \w{\hJ{n}:=I_{n}\setminus I_{n-1}}
(discrete, by assumption) by adding all the maps from any of these
objects into \w[.]{I_{n-1}}

For instance, if \w{I=[\bn]} is the \emph{linear} category of
$n$ composable maps (with degrees as labels):
\[
n~\xra{\phi_{n}}~n-1~\xra{\phi_{n-1}}~\dotsc~2~\xra{\phi_{2}}~\dotsc~1~
\xra{\phi_{1}}~0~,
\]
then \w{I_{k}} consists of the $k$ arrows on the right,
\w[,]{\hJ{k}=\{k\}} and the fine cover thus is \w[.]{J_{k}:=\{\phi_{k}\}}
\end{example}

\begin{example}\label{squarefilter}
If $I$ is the commutative square diagram 

\mydiagram[\label{basicsquare}]{
4 \ar[r]^{d} \ar[d]_{c} & 3 \ar[d]^{b} \\
2 \ar[r]_{a} & 1
}
then \w{\hJ{k}} contains only $k$, while 
\w[,]{J_{2}=\{a:2 \to 1 \}}  \w[,]{J_{3}=\{b:3 \to 1\}} and \w{J_{4}} contains
both \w{c:4 \to 2} and \w{d:4 \to 3} (since \w{I_{3}} contains both $2$ and $3$). 
\end{example}

\begin{remark}\label{rgroup}
As noted in the introduction, a group (or monoid) $G$ may be thought of
as a category with a single object. If we start with a directed indexing 
category \w[,]{I'} and for \w[,]{i\in I'} we add maps \w{g:i\to i} for 
each \w{g\in G} for some group \w{G=G_{i}} (with suitable commutation 
relations with the maps of \w[),]{I'} we obtain a small category $I$ 
(no longer directed) whose diagrams describe directed systems of 
group actions.  Clearly, any orderable cover \w{\cJ'} of \w{I'} induces 
an orderable cover $\cJ$ of $I$.
\end{remark}

\begin{example}\label{egroup}
Let \w{I'} consist of two parallel arrows
\w[,]{\phi_{1},\phi_{-1}:i\to j} \w[,]{G_{i}=\ZZ/2} and \w[.]{G_{j}=0}
Then the indexing category $I$ has a single new non-identity map
\w{f:i\to i} and \w{\phi_{k}\circ f=\phi_{-k}} (\w[).]{k= \pm 1}
Compare \cite{DatuC}.
\end{example}

\begin{mysubsect}{Model categories}
\label{smodel}

Now let $\cC$ be a simplicial model category (cf.\ 
\cite[II, \S 1]{QuiH}), and let \w{\CI} denote the functor category of
$I$-diagrams in $\cC$. There are (at least) two relevant simplicial
model category structures on \w[:]{\CI}

\begin{enumerate}
\renewcommand{\labelenumi}{(\alph{enumi})}
\item For general $I$ and cofibrantly generated $\cC$, we have the
\emph{diagram} model category structure, in which the weak equivalences
and fibrations are defined objectwise, and the cofibrations are generated
(under retracts, pushouts, and transfinite compositions) by the free maps
(free on a generating cofibration at some \w[)]{i\in I} \ -- \
cf.\ \cite[Theorem~11.6.1]{HirM}.
\item If $I$ is a directed indexing category as  above,
it is in particular a (one-sided) Reedy category
(cf.\ \cite[\S 15.1.1]{HirM}). Thus \w{\CI} has a \emph{Reedy} model
category structure, in which the weak equivalences are defined objectwise,
the cofibrations are defined by attaching of a suitable latching object,
and the fibrations are defined by requiring that the structure map to the
matching objects are all fibrations (cf.\ \cite[\S 15.3]{HirM}).
\end{enumerate}
\end{mysubsect}

\begin{remark}\label{rmodel}
In the cases where $I$ is a Reedy category and $\cC$ is cofibrantly
generated, the identity \w{\Id:\cC\to\cC} is a strong Quillen functor
(actually a Quillen equivalence) 
between the two model category structures (see \cite[Theorem~15.6.4]{HirM}), 
considered as a right adjoint from the Reedy model structure to the
diagram model structure.  As a consequence, 
every Reedy fibration is an objectwise fibration  
(cf.\ \cite[Proposition~15.3.11]{HirM}), and conversely, every 
cofibration in the diagram model category is a Reedy cofibration.
In both cases we use the same simplicial mapping spaces
\w[,]{\Map_{\CI}(X,Y)} (sometimes denoted simply by
\w[),]{\map(X,Y)} with
\begin{myeq}[\label{eqmaps}]
\Map_{\CI}(X,Y)_{n}~:=~\Hom_{\CI}(X\times\Delta[n],Y)~.
\end{myeq}
\end{remark}

\begin{mysubsection}{Diagrams over $Z$}
\label{scomma}
For a fixed ground diagram \w[,]{Z:I\to\cC}
the comma category \w{\CIZ} consists of diagrams \w{X:I\to\cC}
over $Z$ \ -- \ that is, for each \w{i\in I} we have maps
\w[,]{p_{i}:X_{i}\to Z_{i}} natural in $I$. Once again \w{\CIZ} has the
two model category structures described above. The simplicial
mapping space \w[,]{\Map_{\CIZ}(X,Y)} defined as in \wref[,]{eqmaps}
will usually be denoted simply by \w[.]{\Map_{Z}(X,Y)} 
We may assume that $Z$ is Reedy fibrant, so in particular (objectwise) fibrant.
\end{mysubsection}

\begin{mysubsection}{Sketchable categories}
\label{ssketch}
Most of our results are valid for quite general simplicial model
categories $\cC$. However, as noted in the introduction, we shall be
mainly interested in the case where \w{\cC=s\cA} is the category of
simplicial objects over some FP-sketchable category $\cA$ (essentially:
a category of (possibly graded) universal algebras  \ -- \ cf.\
\cite[\S 1]{ARosiL}). Note that any such $\cC$ is cofibrantly
generated \ -- \ in fact, a resolution model category (see \cite[\S 3]{BJTurR}). 
Such an $\cA$ will be called \emph{$\fG$-sketchable} if
it is equipped with a faithful forgetful functor to a category of
graded groups (compare \cite[\S 4.1]{BPescF}).
The important property for our purposes is that a map \w{f:X\to Y}
in $\cC$ is a fibration if and only if it is an epimorphism onto 
the basepoint component of $Y$ (cf.\ \cite[II, \S 3, Prop.~1]{QuiH}).

If we let \w[,]{\cA=\Gp} we obtain the homotopy category of
pointed connected topological spaces (see \cite[V, \S 6]{GJarS}), so
our assumptions cover all the topological applications mentioned in
the introduction.

In this context we may need to consider diagrams over a fixed ground
diagram $Z$: following \cite[\S 2]{QuiC} and \cite[\S 3]{BecT},  
for (diagrams of simplicial objects in) a $\fG$-sketchable category $\cA$,
one may identify $Z$-modules with abelian group objects over $Z$.
Thus we may be forced to work in \w{\CIZ} if we want to study cohomology
with twisted coefficients.
\end{mysubsection}

\begin{mysubsection}{Diagram completion}
\label{scompl}
Any inclusion of categories \w{J\hookrightarrow I} induces a
forgetful \emph{truncation} functor
\w[,]{\tau=\tau^{I}_{J}:\CI\to\CJ} and this has a right
adjoint \w[,]{\xi=\xi^{I}_{J}:\CJ\to\CI} which assigns to a
diagram \w{Y:J\to\cC} the diagram \w{\xi Y:I\to\cC} with
\w{\xi Y(i):=\lim_{i/J} Y} for each \w{i\in I} (where \w{i/J}
is the obvious subcategory of the under category \w[).]{i/I}
Note that \w{\xi Y(j)=Y_{j}} for \w[.]{j\in J} Also, if
\w{J\subseteq J' \subseteq I} then
\w[,]{\xi^{J'}_{J}=\tau_{J'}^{I}\circ\xi^{I}_{J}} 
\w[,]{\xi^{I}_{J}=\xi^{I}_{J'}\circ\xi^{J'}_{J}} and
\w[,]{\tau^{I}_{J}=\tau^{J'}_{J}\circ\tau^{I}_{J'}}
so we shall often omit the superscripts from these functors, with the second 
category understood from the context.

The resulting monad \w{\sigma_{J}:=\xi_J\circ\tau_J:\CI\to\CI}
is called the \emph{completion} at $J$, and we denote the
augmentation of the adjunction by \w[.]{\omega_{J}:Y\to\sigma_J Y}

Moreover, given a fixed \w[,]{Z\in\CI} the truncation functor
\w{\hat{\tau}_{J}:\CI/Z\to\CJ/\tau Z} also has a right adjoint
\w[,]{\hat{\xi}_{J}:\CJ/\tau Z\to\CIZ} with the limit
\w{\hat{\xi}_{J} Y(i):=\lim_{i/J} Y} taken over \w{\tau_{J}Z}
(that is, the diagram whose limit we take consists of \w{Y\rest{i/J}} 
mapping to \w[,]{\tau_{J}Z} where the latter includes also \w[).]{Z_{i}}
Thus the completion at $J$ in \w{\CIZ} is:
%
\begin{myeq}\label{ezero}
\hsi_{J}Y(j)~=~\sigma_{J}Y(j)\,\times_{\sigma_{J}Z(j)}~Z_{j}~,
\end{myeq}
\noindent where the structure map
\w{\sigma_{J}q:\sigma_{J}Y\to\sigma_{J}Z} is induced by the
functoriality of limits.  Once again, there will be an augmentation
\w[.]{\hoa_J:Y \to \hsi_J Y}
\end{mysubsection}

\begin{example}\label{ecompl}
If \w{I=[\bn]} is linear (\S \ref{emainex}) and \w{J=[\bk]} is an
initial (right) segment, then for any tower \w{Y:[\bn]\to\cC} we have:
\[
\sigma_J Y(i)~=~\begin{cases}
Y_{i} & \text{if}\ i\leq k\\
Y_{k} & \text{if}\ i\geq k
\end{cases}
\]
\end{example}

\begin{example}\label{ecomplsquare}
If $I$ is the commutative square of \S \ref{squarefilter}, then
\w{\sigma_{I_{3}} Y} is the pullback diagram
\mydiagram{
Y_{2} \times_{Y_{1}} Y_{3}  \ar[r] \ar[d] & Y_{3} \ar[d]^{Y(b)} \\
Y_{2} \ar[r]_{Y(a)} & Y_{1}   ,
} 
while \w{\hsi_{I_{3}}Y(3)} is the further pullback 
\mydiagram{
\hsi_{I_{3}}Y(3) \ar[r] \ar[d] & Y_{2} \times_{Y_{1}} Y_{3} \ar[d] \\
Z_{4}  \ar[r] & Z_{2} \times_{Z_{1}} Z_{3} .
}
\end{example}

\begin{example}\label{ecomplmatch}
If \w{I=\Delta'\subseteq \Delta^{\op}} is the indexing category for
restricted simplicial objects $Y$ (without degeneracies), and $J$ is its
truncation to dimensions $<n$, then \w{\sigma_J Y(n)=M_{n}Y} is the
classical matching object of \cite[X,\,\S 4.5]{BKaH}
\end{example}

\begin{mysubsection}{Maps of diagrams}
\label{smaps}
Given a fixed Reedy fibrant ground diagram \w[,]{Z:I\to\cC} consider
the simplicial mapping space \w{\Map_{Z}(X,Y)} as in \S \ref{scomma}
for \w[,]{X,Y\in\CI/Z} where $X$ is cofibrant and  $Y$ is fibrant.

In the cases of interest to us, $Y$ will be an abelian group object
in \w[,]{\CI/Z} so the homotopy groups of $\Map_{Z}(X,Y)$ are the cohomology
groups of $X$ with coefficients in $Y$ (see \cite[\S 5]{BJTurR} for
further details). In order to build our restriction tower, 
we need an appropriate orderable
cover $\cJ$ of $I$ (\S \ref{dcover}), yielding a filtration \
\[
I\supseteq \dotsc\supseteq\ I_{n}\ \supseteq\ I_{n-1}\ \supseteq\ \dotsc~.
\]

Let \w{M_{n}:=\Map_{\CIn{I_{n}}/\tau_{n}Z}(\tau_{n}X,\tau_{n}Z)} for each
\w[,]{n\in N} where \w{\tau_{n}X}  is the restriction of a diagram 
\w{X\in\CI} to \w[.]{I_{n}} The inclusions \w{I_{n-1}\hra I_{n}} and 
\w{I_{n}\hra I} induce maps \w{\rho_{n}:M_{n}\to M_{n-1}} and
\w{\hat{\rho}_{n}:M\to M_{n}} which fit into a tower:

\mydiagram[\label{eone}]{
\Map_{Z}(X,Y) \ar[dr]^>>>>{\hat{\rho}_{n+1}} \ar@/^1pc/[drr]^>>>>{\hat{\rho}_{n}}
\ar@/^2pc/[drrr]^>>>>{\hat{\rho}_{n-1}} & & & & \\
\ldots \ar[r] & ~~M_{n+1}~~~~ \ar[r]^{\rho_{n+1}} &
~~M_{n}~~~~ \ar[r]^{\rho_{n}} &
~~M_{n-1}~~~~ \ar[r]^{\rho_{n-1}} & ~~~~\ldots M_{0}
}
\noindent with
%
\begin{myeq}\label{etwo}
\Map_{Z}(X,Y)~\cong~\lim_{n}\,M_{n}~.
\end{myeq}
\end{mysubsection}

%
%
\section{A tower of fibrations}
\label{ctow}

To determine when \wref{eone} is a tower of fibrations (so that
\wref{etwo} is a homotopy limit), we need the following:

\begin{defn}\label{dgood}
Let $I$ be an indexing category, $\cC$ a model category,
and \w[.]{Z\in\CI} Given an orderable cover
\w{\cJ=\{J_{\nu}\}_{\nu \in N}} of $I$ with associated filtration
\w[,]{(I_{n})=(\fJ{n})_{n\in \ZZ}} let \w{\tau_{k}:\CI\to \CIn{I_k}} and
\w{\tau^{m}_{k}:\CIn{I_m}\to \CIn{I_k}} denote the truncation functors, with
adjoints indexed accordingly. A diagram \w{Y\in\CIZ} is called
$\cJ$-\emph{fibrant} if for each \w[,]{n\in \ZZ} the augmentation
\w{\hoa_{n+1}:\tau_{n+1}Y\to\hsi_{n}^{n+1} Y=\hsi^{I_{n+1}}_{I_{n}}Y} is a
fibration in 
\w[.]{\CIn{I_{n+1}}/\sigma_{n}^{n+1}Z=\CIn{I_{n+1}}/\sigma^{I_{n+1}}_{I_{n}}Z}
\end{defn}

\begin{remark}\label{assume}
Because we assumed the  degree is strictly decreasing, \w{I_{n+1}} and
$I$ are the same so far as the augmentation map \w{\hoa_{n+1}} is concerned.
Thus if we assume for simplicity that \w[,]{I=I_{n+1}} then
\w{\hoa_{n+1}} may be identified with its adjoint map \w{Y\to\hsi_{n}Y} 
in \w[.]{\CIn{I_{n+1}}/\sigma_{n}^{n+1}Z=\CI/\sigma_{n}Z}
\end{remark}

%
%
\begin{prop}\label{pone}
Assume \w{\cJ=\{J_{\nu}\}_{\nu \in N}} is an orderable cover of $I$,
\w{X\in\CIZ} is cofibrant, and \w{Y\in\CIZ} is a $\cJ$-fibrant abelian
group object. Then 
$$
F_{n+1}~\to~M_{n+1}~\xra{\rho_{n+1}}~M_{n}
$$
is a fibration sequence of simplicial abelian groups for each
\w[,]{n\in \ZZ} and the fiber $F_{n+1}$ is
\w[.]{\Map_{\CInZ{J_{n+1}}\rest{J_{n+1}}}\,(X\rest{J_{n+1}},~\Fib(\omega_{n+1}))}
Here \w{\Fib(\omega_{n+1})} denotes the fiber (in 
\w[)]{\CIn{I_{n+1}}/\sigma_{n}^{n+1} Z} of the augmentation 
\w[.]{\omega_{n+1}:\tau_{n+1}Y\to\sigma_{n}^{n+1} Y=\sigma_{I_{n}}^{I_{n+1}}\,Y}
\end{prop}
\begin{proof}
Assume for simplicity that \w[,]{I=I_{n+1}(=\fJ{n+1})} with
\w{\tau_{n}=\tau_{I_{n}}:\CI\to\CIn{I_{n}}} and \w{\sigma_{n}(=\sigma_{\fJ{n}})} 
the completion at \w{I_{n}(=\fJ{n})} (as in Remark \ref{assume}).
Then there is a natural adjunction isomorphism:
$$
\Map_{\CIn{I_{n}}/\tau_{n} Z}(\tau_{n} X,\tau_{n} Y)~
=~\Map_{\CI/\sigma_{n} Z}(X,\hsi_{n} Y)~,
$$
\noindent under which \w{\rho_{n}} is identified with the map induced
in \w{\Map_{\sigma_{n}Z}(X,-)} by \w[.]{\hoa_{n+1}:Y\to\hsi_{n} Y}
This \w{\hoa_{n+1}} is a fibration in \w{\CI/\sigma_{n} Z} by Definition
\ref{dgood}, and thus induces a fibration of mapping spaces, with fiber
\w[.]{\Map_{\sigma_{n} Z}(X,~\Fib(\hoa_{n+1}))}

Thus, it suffices to identify the fiber instead as 
\w[.]{\Map_{\CInZ{J_{n+1}}\rest{J_{n+1}}}\,(X\rest{J_{n+1}},~\Fib(\omega_{n+1}))}
However, since \w{\hoa_{n+1}(i):Y_{i}\to\hsi_{n} Y(i)} is the identity for
\w[,]{i\in I_{n}} \  the diagram \w{\Fib(\hoa_{n+1}):I\to\cC} is trivial
(over $Z$) when restricted to \w[,]{I_{n}} and since $\cJ$ was
orderable, any map \w{f:X=\tau_{n+1}X\to\Fib(\hoa_{n+1})} is
determined uniquely by its restriction to \w{J_{n+1}} \ -- \ in fact,
to the discrete subcategory \w[.]{\hJ{n+1}:=J_{n+1}\setminus I_{n}}

The fact that $Y$ is an abelian group object in \w{\CIZ}
implies, by definition, that for each \w{i\in I} there is a commuting triangle:
%
\mydiagram[\label{eeight}]{
Z_{i}  \ar[d]_{=} \ar[r]^{s_{i}} & Y_{i}  \ar[dl]^{q_{i}}\\ 
Z_{i}  & \  ,
}
\noindent natural in $I$. Thus \w{\Fib(\hoa_{n+1})(j)} for \w{j\in J_{n+1}} 
is by definition the pullback of:
\mydiagram{
 & Z_{j} \ar[d] \ar[rd]_{\sigma_{n}s_{j}\circ\omega_{Z}} \ar[rrd]^{\Id} & & \\
Y_{j}~\ar[r]^<<<<{\hoa}_<<<<{(\omega,q_{j})}~~ &
~~~~\hsi_{n} Y_{j}~~~~~~~~=~&
\sigma_{n} Y(j)\hspace*{3mm}\times_{\sigma_{n} Z(j)} & Z_{j}~,
}
\noindent and we readily check that this is the same as
\w[,]{\Fib(\omega_{n+1})(j)} which is the pullback of:
\mydiagram{
 & \sigma_{n} Z(j)\ar[d]^{\sigma_{n} s(j)} \\
Y_{j} \ar[r]^-{\omega_Y}& \sigma_{n} Y(j)  \ .
}
\end{proof}

\begin{mysubsect}{Directed indexing diagrams}
\label{sdid}

We shall now see how Proposition \ref{pone} applies when $\cJ$ is an
orderable cover of a directed indexing category $I$ (see \S \ref{sdirect}).

Recall that in the Reedy model category structure (cf.\ \S \ref{smodel})
on \w[,]{\CI} a map \w{f:X\to Y} is a fibration if and only if
%
\begin{myeq}\label{ethree}
X_{j}~\xrightarrow{(f,p)}~Y_{j}\times_{\sigma_{n} Y(j)}~\sigma_{n} X(j)
\end{myeq}
\noindent is a fibration in $\cC$ for every \w{j\in\Obj I} with
\w[,]{\deg(j)=n+1} where \w{\sigma_{n}=\sigma_{I_{n}}} is the completion
at \w[.]{I_{n}} In \w{\CIZ} we must replace $\sigma_{n}$ by $\hsi_{n}$
(\S \ref{scompl}), of course.
\end{mysubsect}

%
%
\begin{lemma}\label{lone}
If $I$ is a directed indexing category, any Reedy fibrant
\w{Y\in\CIZ} is $\cJ$-fibrant for the fine cover of $I$
(\S \ref{emainex}).
\end{lemma}

\begin{proof}
Once again we assume \w{I=I_{n+1}} (\S \ref{assume}), so we must show
that \w{\hoa_{n+1}:Y\to\hsi_{n} Y} is a fibration in
\w[.]{\CI/\sigma_{n} Z} Since $\hoa_{n+1}$ is the identity for
\w[,]{j\in I_{n}} consider \w[.]{j\in\hJ{n+1}:=I_{n+1}\setminus I_{n}}
Since $Y$ is Reedy fibrant in \w[,]{\CIZ} \w{q:Y\to Z} is a Reedy fibration in
\w[,]{\CI} and since $\cJ$ is fine, this means that
$$
Y_{j}~\xra{(\omega_{n+1},q_{j})}~\sigma_{n}Y(j)\times_{\sigma_{n}Z(j)} Z_{j}~=~
\hsi_{n}Y(j)~=~\hsi_{n}Y(j)\times_{\hsi_{n}Y(j)}~\hsi_{n}Y(j)
$$
is a fibration in $\cC$ \ -- \ which shows that \wref{ethree}
indeed holds for each \w[.]{j\in I}
\end{proof}

%
%
\begin{prop}\label{ptwo}
Let \w{\cC=s\cA} for some $\fG$-sketchable category $\cA$  (\S \ref{ssketch}),
and let \w{\cJ=\{J_{\nu}\}_{\nu \in N}} be an orderable cover of
a directed indexing category $I$, with \w{Z \in \CI} Reedy fibrant. 
Then any abelian group object \w{Y\in\CIZ} is weakly equivalent to a
fibrant (objectwise) abelian group object which is $\cJ$-fibrant. 
\end{prop}

\begin{proof}
Because $I$ is directed, we may construct the desired $\cJ$-fibrant
replacement $\bY$ \ -- \ an abelian group object in \w{\CIZ} \ -- \
by induction on the degree of \w[.]{j\in I} 
Moreover, we assumed that $Z$ is Reedy fibrant, so in particular 
objectwise fibrant (see Remark \ref{rmodel}).  Note that any abelian
group object \w{p:V\to Z} in \w{\CIZ} is (objectwise) fibrant, since
$p$ has a section by \wref{eeight} and \S \ref{ssketch}; hence $p$ has the
right lifting property with respect to any acyclic cofibration.

We assume by induction on \w{\deg(j)=n+1} that both
\w{\bar\omega_{n+1}(j):\bY_{j}\to\hsi_{n} \bY(j)} and
\w{\bar{q}_{j}:\bY_{j}\to Z_{j}} are fibrations in $\cC$.  Since for 
each $j$, \w{\sigma_{n} Y(j)} is defined as a limit, and an
abelian group object structure on any $V$ is a map \w{V\times_{Z}\, V\to V} 
(over $Z$), by functoriality (and commutativity) of limits we see that 
\w{\sigma_{n} q:\sigma_{n} \bY \to\sigma_{n} Z} is an abelian group object, too \ -- \ 
so \w{\sigma_{n} q} is an objectwise fibration in \w[.]{\CI} But
$$
\xymatrix@R=25pt{
\hsi_{n} \bY_{j} \ar[r]^{\pi_{Z}} \ar[d] & Z_{j} \ar[d]\\
\sigma_{n} \bY(j)\ar[r]^{\sigma_{n} q} & \sigma_{n} Z(j)
}
$$
is a pullback square, by definition, so \w{\pi_{Z}} is a fibration
in $\cC$ by base change. 

In the induction step, for each $j$ of degree \w[,]{n+1} we factor:
$$
\bar{\hoa}_{j}:\bY_{j}~\to~\hsi_{n} \bY(j)~=~
\sigma_{n} \bY(j)\,\times_{\sigma_{n} Z(j)}~Z_{j}
$$
as 
$$
\bY_{j}~\hra ~\bY_{j}'~\xra{\bar{\omega}_{j}'}~\hsi\bY(j)
$$
(an acyclic cofibration followed by a fibration), and replace
\w{\bY_{j}} by \w[.]{\bY'_{j}} Both \w{\bar{\omega}'_{j}}  and
\w{\bar{q}_{j}:=\pi_{Z}\circ\bar{\omega}'_{j}:\bY_{j}\to Z_{j}} are then
fibrations in $\cC$, as required.
\end{proof}

\begin{remark}\label{rcgroup}
This actually works for some orderable covers of indexing categories
which are not directed. For example, if we use the fine cover $\cJ$ for
an indexing category $I$ constructed as in \S \ref{rgroup}, we can
still change any $Y$ into a $\cJ$-fibrant one by induction on the degree
in \w[,]{I'} since we have not introduced any new objects
\end{remark}

\begin{example}\label{ecgroup}
In Example \ref{egroup}, for any \w[,]{Y\in\CI} \w{\sigma Y} is given by:
$$
\sigma Y(j)~=~Y_{i}\times Y_{i} 
\raisebox{-0.6ex}
{$~~~\stackrel{\textstyle \longrightarrow}{\longrightarrow}~~~$}
Y_{i} ~=~\sigma Y(i)~,
$$
with horizontal maps \w{Y(\phi_{\pm1})} the two projections, and 
\w{f:\sigma Y(j)\to \sigma Y(j)} the switch map. To make this
$\cJ$-fibrant for the obvious (fine) cover, we just have to choose \w{\bY}
so that \w{\hoa:\bY_{j}\to \sigma \bY(j)} is a \ $\ZZ/2$-equivariant
fibration\vsn. 
\end{example}

\begin{mysubsect}{The dual construction}
\label{sdc}

The approach described above is clearly best suited to directed indexing
categories $I$ where the degree function is non-negative. In the inverse
case, the dual approach may be preferable:

Given a small indexing category $I$ and a subcategory $J$,
the truncation functor \w{\tau=\tau^{I}_{J}:\CI\to\CJ} also
has a left adjoint \w[,]{\zeta=\zeta^{I}_{J}:\CJ\to\CI}
which assigns to a diagram \w{X:J\to\cC} the diagram
\w{\zeta X:I\to\cC} with \w{\zeta X(i):=\colim_{J/i} X} for
each \w[.]{i\in I} We denote the resulting comonad on \w{\CI} by
\w[.]{\theta_J=\zeta_J\circ\tau_J} Note that if \w[,]{X\in\CIZ} then
\w{\theta_J X} comes equipped with a map to \w[,]{\theta_J Z\in\CIZ}
so we do not need the analogue of \wref[.]{ezero}

We then say that a diagram \w{X\in\CIZ} is $\cJ$-\emph{cofibrant}
for an orderable cover $\cJ$ if for each \w[,]{n\in \ZZ} the coaugmentation 
\w{\eta_{n+1}:\theta_{n}^{n+1}X=\theta_{I_{n}}^{I_{n+1}} X\to\tau_{n+1}X} 
is a cofibration in \w[.]{\CIn{I_{n+1}}/\tau_{n+1}Z} We then have:

%
%
\begin{prop}\label{dualtower}
Assume \w{\cJ=\{J_{\nu}\}_{\nu \in N}} is an orderable cover of $I$,
\w{X\in\CIZ} is $\cJ$-cofibrant, and \w{Y\in\CIZ}
is a fibrant abelian group object. Then
$$
F_{n+1} \to \Map_{\CIn{I_{n+1}}/\tau_{n+1}Z}(\tau_{n+1}X,\tau_{n+1}Y)~
\xra{\rho_{n+1}}~\Map_{\CIn{I_{n}}/\tau_{n}Z}(\tau_{n}X,\tau_{n}Y)
$$
is a fibration sequence of simplicial abelian groups for each
\w[,]{n\in \ZZ} and the fiber \w{F_{n+1}} is
\w[.]{\Map_{\CInZ{J_{n+1}}\rest{J_{n+1}}}\,(\Cof(\eta_{n+1}),\,Y\rest{J_{n+1}})}

Here \w{\Cof(\eta_{n+1})} denotes the cofiber (over \w[)]{\tau_{n+1} Z} of the 
coaugmentation \w[.]{\eta_{n+1}:\theta_{n}^{n+1} X\to \tau_{n+1} X}
\end{prop}

\begin{proof}
Dual to that of Proposition \ref{pone}
\end{proof}

Note that if $I$ is a directed indexing category, we need no special
assumptions on \w{X,Y\in\CIZ} (or $\cC$) in order for the dual of
Proposition \ref{ptwo} to hold, since all colimits are over $Z$ to
begin with.  Thus, we can again build $\cJ$-cofibrant
replacements by induction on degree to yield the following:  

%
\begin{prop}\label{ptwoagain}
Let \w{\cC=s\cA} for some $\fG$-sketchable category $\cA$,
and let \w{\cJ=\{J_{\nu}\}_{\nu \in N}} be an orderable cover of
a directed indexing category $I$. Then any \w{X\in\CIZ} is weakly
equivalent to a cofibrant object (with respect to the model structure
of \S \ref{smodel}(a)), which is $\cJ$-cofibrant.
\end{prop}

\end{mysubsect}

%
%
\section{The two truncation spectral sequences}
\label{css}

As noted above, for a suitable model category $\cC$ and any indexing
category $I$, given \w{Z\in\CI} and \w{X,Y\in\CIZ} with $X$
cofibrant and $Y$ a fibrant abelian group object, the homotopy groups
of \w{\Map_{Z}(X,Y)} are the cohomology groups \w{H^{\ast}(X/Z,Y)}
(suitably indexed). Thus if $\cJ$ is some orderable cover of $I$ 
such that $Y$ is $\cJ$-fibrant, the homotopy spectral sequence 
for the tower of fibrations (cf.\ \cite[VII, \S 6]{GJarS}) of (fibrant)
simplicial sets \wref{eone} yields a spectral sequence with 
\w[.]{E^{2}_{k,n}~=~\pi_{k+n}\Fib(\rho_{n})~\Longrightarrow 
\pi_{k+n} \Map_{Z}(X,Y)}
To identify the \ $E^{2}$-term, we need the following:

\begin{defn}\label{drelc}
Consider an orderable cover \w{\cJ=\{I',J\}} of a diagram $I$
(where we have in mind \w[,]{I=I_{n+1}} \w[,]{I'=I_{n}}
and \w[).]{J=J_{n+1}} If $Y$ is an abelian group object in
\w{\CIZ} which is $\cJ$-fibrant, then we have a fibration sequence
\[
\Fib(\hoa)~\to~Y~\xra{\hoa}~\hsi Y~,
\]
of abelian group objects over $Z$, where $\hsi$ is the completion
at \w[.]{I'}

We define the \emph{relative cohomology} of the pair \w{(I,J)} to be 
the total left derived functor of \w[,]{\Hom_{\CJ/Z\rest{J}}(-,~\Fib(\hoa))} 
(into simplicial abelian groups), denoted by \w[.]{H(X/Z;\hoa)} 
In particular, the $i$-th \emph{relative cohomology group} for
\w{(I,J)} is  \w[.]{H^{i}(X/Z;\hoa):=\pi_{i}H(X/Z;\hoa)} 
\end{defn}

\begin{remark}\label{rindex}
Note that in most applications the abelian group object \w{Y\in\CIZ} 
will be an $n$-th dimensional Eilenberg-Mac Lane object (over $Z$), 
in which case it is customary to re-index the relative cohomology
groups so that \w[.]{H^{n}(X/Z;\hoa):=\pi_{0}H(X/Z;\hoa)}

Observe, however, that our setup allows $Y$ to consist of
Eilenberg-Mac Lane objects of varying dimensions, with the maps 
\w{Y(f)} representing cohomology operations. In this general setting,
no canonical re-indexing exists.
\end{remark}

\begin{fact}\label{frelc}
Given \w{I,J,I'} and \w{Y,Z} as above, for any (cofibrant)
\w{X\in\CIZ} there is a long exact sequence in cohomology
\begin{myeq}[\label{eqrelc}]
\to H^{i}((X/Z)\rest{J};\,\hoa)\to
H^{i}(X/Z;\,Y)\to H^{i}((X/Z)\rest{I'};Y\rest{I'})\to
H^{i+1}((X/Z)\rest{J};\,\hoa)\to
\end{myeq}
\end{fact}

\begin{thm}\label{tone}
For any simplicial model category $\cC$, directed indexing category $I$,
and diagrams \w[,]{Z:I\to\cC} \w[,]{X\in\CIZ}
abelian group object \w[,]{Y\in\CIZ} and left-orderable
cover $\cJ$ of $I$ there is a first quadrant spectral sequence with:
\[
E^{2}_{s,t}~=~H^{t+s}((X/Z)\rest{J_{t}};\,\hoa)
~\Longrightarrow H^{s+t}(X/Z;\,Y)
\]
and \w[.]{d^{2}:E^{2}_{s,t}\to E^{2}_{s-2,t+1}}
\end{thm}

\begin{proof}
Replace $Z$ by a weakly equivalent Reedy fibrant diagram in 
\w[,]{\CI} then $X$ by a weakly equivalent cofibrant object in \w[,]{\CIZ}
and then using Proposition \ref{ptwo} to replace $Y$ by a weakly 
equivalent $\cJ$-fibrant abelian group object in \w[.]{\CIZ} 
Proposition \ref{pone} then implies that \wref{eone} is a tower of
fibrations, and the associated homotopy spectral sequence has the
specified relative cohomology groups as the homotopy groups of the
fibers (which are the \ $E^{2}$-term of the spectral sequence, in our
indexing). 
\end{proof}

The spectral sequence need not converge, in general, without some
cohomological connectivity assumptions on the subdiagrams (unless the
cover $\cJ$ is finite, of course).

\begin{remark}\label{rfine}
If $\cJ$ is the fine cover, the \ $E^{2}$-term simplifies to:
$$
E^{2}_{s,t}~=~\prod_{j\in\hJ{t}}~H^{t+s}(X_{j}/Z_{j},\hat{\phi}_{j})~,
$$
where \w{\hat{\phi}_j:Y_{j}\to \underset{\phi:j\to i}{\lim}\,Y_{i}} 
is the structure map.
\end{remark}

Using the approach of \S \ref{sdc}, we also obtain a dual spectral sequence:

\begin{thm}\label{ttwo}
For $\cC$, $I$, $Z$, $X$, and $Y$ as in Theorem \ref{tone}, and $\cJ$
right-orderable, there is a first quadrant spectral sequence with:
$$
E^{2}_{s,t}~=~H^{s+t}(\eta_{t};\,Y)
~\Longrightarrow H^{s+t}(X/Z;\,Y)~.
$$
\end{thm}

\begin{remark}\label{rdual}
Note that 
\w{H^{\ast}(\eta_{t};\,Y):=H^{\ast}(\Cof(\eta_{t})/Z\rest{J_{t}};\,Y)}
is just the usual cohomology of the map of diagrams 
\w{\eta_{t}:\theta_{t-1}^{t} X\to\tau_{s}X} (see \S \ref{sdc}).
This fits into the usual long exact sequence of a pair, 
dual to that of \wref[.]{eqrelc}

When $X$ is cofibrant, $Z$ and $Y$ are constant, and
\w{\colim_{I} X=\hocolim_{I} X} -- \ for example, when $I$ is a
partially ordered set, so \w{\colim_{I} X=\bigcup_{i\in I}\,X_{i}} -- \
then \w[,]{H^{\ast}(X/Z;\,Y)=H^{\ast}(\colim_{I} X/Z;\,Y)} and the dual
spectral seqeunce is simply the usual Mayer-Vietoris spectral sequence
for the cover $X$ of \w{\colim_{I} X} (cf.\ \cite[\S 5]{SegCC}, and
compare \cite[XII,\,4.5]{BKaH}, \cite[\S 10]{VogtHL}, and \cite{SlomS}).
\end{remark}

\begin{example}\label{egsquare}
Let $I$ be the commuting square as in Example \ref{squarefilter}:

Given a diagram of abelian group objects \w[,]{Y:I\to\cC} 
the successive fibers \w{\Fib(\omega_{n+1})} (see Proposition
\ref{pone}) are: 
\[
\xymatrix{
\Ker(Y(c))\cap\Ker(Y(d)) \ar[r]\ar[d] & 0 \ar[d] \\
0\ar[r] & 0 
}
\]
for \w[;]{\omega_{4}:Y=\tau_{4}Y\to\sigma_{3}Y}
\[
\xymatrix{
& \Ker(Y(b)) \ar[d] \\ 0 \ar[r] & 0
}
\]
for \w[;]{\omega_{3}:\tau_{3}Y\to\sigma_{2}Y}
\[\xymatrix{\Ker(Y(a)) \ar[r] & 0}\]
for \w[;]{\omega_{2}:\tau_{2}Y\to\sigma_{1}Y} and the single object
\w{Y_{1}} for \vsm \w[.]{\omega_{1}:\tau_{1}Y\to\sigma_{0}Y}

Thus the \ $E^{2}$-term for the spectral sequence consists of only
four non-trivial lines: 
\begin{myeq}[\label{eqsqur}]
E^{2}_{s,t}~\cong~\begin{cases} 
H^{s+4}(X_{4};~\Ker(Y(c))\cap\Ker(Y(d)))  & \text{if} \ t=4;\\
H^{s+3}(X_{3};~\Ker(Y(b)))               & \text{if} \ t=3;\\
H^{s+2}(X_{2};~\Ker(Y(a)))               & \text{if} \ t=2;\\
H^{s+1}(X_{1};~Y_{1})                   & \text{if} \ t=1;\\
0 & \text{otherwise.}
\end{cases}
\end{myeq}

If we had used the fine cover, by Remark \ref{rfine} we would instead have:
\[
E^{2}_{s,t}~\cong~\begin{cases} 
H^{s+3}(X_{4};~\Ker(Y(c))\cap\Ker(Y(d)))  & \text{if} \ t=3;\\
H^{s+2}(X_{3};~\Ker(Y(a)))\oplus H^{s+2}(X_{2};~\Ker(Y(b))) & \text{if} \ t=2;\\
H^{s+1}(X_{1};~Y_{1})                  & \text{if} \ t=1;\\
0 & \text{otherwise.}
\end{cases}
\]
\end{example}

\begin{remark}\label{rsquare}
The square can be thought of as a single morphism in the category of arrows, 
so that we could analyze it as in \cite[\S 4]{BJTurR}, where 
\w{H^{\ast}(X;Y)} is shown to fit into a long exact sequence with ordinary
cohomology groups in the remaining two slots. See \S \ref{scompar} below.
\end{remark}

%
%
\section{An approach through local cohomology}
\label{calc}

The towers of Section \ref{ctow} were constructed by covering a
given indexing category $I$ by truncated subcategories, obtained by
omitting successive initial (or terminal) objects. We now present
an alternative approach, using subcategories obtained by omitting
\emph{internal} objects of $I$. As we shall see, the resulting towers 
differ in nature from those considered above.

\begin{defn}\label{dgoodi}
An indexing category $I$ will be called \emph{strongly directed} if:
\begin{enumerate}
\renewcommand{\labelenumi}{\roman{enumi}.~~}
\item It is \emph{directed} in the sense of having no maps \w{f:i\to i}
but the identity.
\item It has a nonempty \emph{weakly initial} subcategory (necessarily
  discrete) consisting of all objects with no incoming maps, as well
  as a nonempty \emph{weakly final} subcategory consisting of all
  objects with no outgoing maps. 
\item It is \emph{locally finite} (that is, all \ $\Hom$-sets are finite).
\item $I$ (that is, its underlying undirected graph) is \emph{connected}.
\end{enumerate}
\end{defn}

\newcommand{\tupleis}[1][ ]{the tuple \w{\tuple} is admissible{#1}}

\begin{defn}\label{admit}
We refer to \w{\tuple} as \emph{admissible} if:
\begin{enumerate}
\renewcommand{\labelenumi}{(\alph{enumi})~~}
\item \ $\cC$ is a simplicial model category;
\item \ $I$ is \good[;]
\item \w{Z\in\CI} is Reedy fibrant (hence objectwise fibrant);
\item \w{X,Y\in\CIZ} with $X$ cofibrant and $Y$ a fibrant abelian
  group object. 
\end{enumerate}
\end{defn}

\begin{defn}\label{dtrans}
For any categories $\cC$ and $I$ and diagrams \w{Z\in\CI} and
\w[,]{X,Y\in\CIZ} the product of simplicial sets
$$
\trans{\CIZ}{X}{Y}~:=~\prod_{i\in I}\Map_{\cC/Z_{i}}(X_{i},Y_{i})~.
$$
will be called the \emph{space of discrete transformations} from $X$
to $Y$ over $Z$.

We shall generally abbreviate this to \w[.]{\trans{Z}{X}{Y}}
Note that these are maps of functors only for the discrete indexing
category \w[,]{I^{\delta}} with no non-identity maps.
\end{defn}

\begin{mysubsect}{The primary tower}
\label{spt}

In the spirit of Section \ref{ccd}, for any finite indexing category $I$ we
construct a finite sequence of full subcategories 
\begin{myeq}[\label{eqprimt}]
I_{1} \subset I_{2} \subset \dots I_{n}=I
\end{myeq}
of $I$, starting with \w[,]{I_{1}} whose objects are the weakly
initial and final sets. 

As before, this can be done in several ways (ultimately yielding
variant spectral sequences). In any case, we can refine \wref{eqprimt}
so that for each $k$, \w{I_{k-1}} is obtained from \w{I_{k}} by
omitting a single internal object \w{i_{k}} (where \emph{internal}
means that it is neither weakly initial nor weakly final).

If \w{\tuple} is admissible, the inclusions of categories
\w{\iota_{k-1}:I_{k-1}\hra I_{k}} induce a finite tower of simplicial
abelian groups:   
\begin{myeq}[\label{eqprimary}]
\Map_{\CInZ{I_{n}}}(X,Y)~\to\dotsc\to~
\Map_{\CInZ{I_{k}}}(X,Y)~\xra{\iota_{k-1}^{\ast}}~
\Map_{\CInZ{I_{k-1}}}(X,Y)~\to\dotsc~,
\end{myeq}
analogous to \wref[.]{eone}
\end{mysubsect}

\begin{mysubsection}{The auxiliary fibration}
\label{sauxf}
Unfortunately, \wref{eqprimary} is not, in general, a tower of
fibrations, so we cannot use it directly to obtain a useable spectral
sequence for the cohomology of a diagram. To do so, we must replace
it (up to homotopy) by a tower of fibrations, with 
\w{\Map_{Z}(X,Y)} as its homotopy inverse
limit. The resulting spectral sequence (abutting to the homotopy
groups of \w[),]{\Map_{Z}(X,Y)} will have the homotopy groups of the
homotopy fibers of the maps \w{\iota_{k}^{\ast}} as its  \ $E^{2}$-term. In
fact, instead of constructing the replacement directly, we make use of
the following observation: 

For any indexing category $I$ and diagrams \w[,]{X,Y:I\to\cC} the set
\w{\nat{\CI}{X}{Y}} of diagram maps (natural transformations) from $X$
to $Y$ fits into an equalizer diagram:
\begin{myeq}[\label{eqtwenty}]
\nat{\CI}{X}{Y}\hra \prod_{i\in I}\Hom_{\cC}(X_{i},Y_{i}) \equaliz
\prod_{i,j\in I}~\prod_{\eta\in\Hom_{I}(i,j)}~\Hom_{\cC}(X_{i},Y_{j})~.
\end{myeq}
\noindent Here the two parallel arrows map to each factor indexed
by \w{\eta:i\to j} in $I$ by the appropriate projection, followed by either
\w[,]{Y(\eta)_{\ast}:\Hom_{\cC}(X_{i},Y_{i})\to\Hom_{\cC}(X_{i},Y_{j})} or \
\w[,]{X(\eta)^{\ast}:\Hom_{\cC}(X_{j},Y_{j})\to\Hom_{\cC}(X_{i},Y_{j})}
respectively.

In the case where $Y$ is an abelian group object in \w{\CI} 
(or \w[),]{\CIZ} this describes \w{\nat{\CI}{X}{Y}} as the kernel of the
difference $\xi$ of the two parallel arrows. By considering mapping
spaces rather than \ $\Hom$-sets, we obtain a left-exact sequence of
simplicial abelian groups:
\begin{myeq}[\label{eqtwentyone}]
0\to\map(X,Y)\to~\trans{}{X}{Y}~\xra{\xi}
\prod_{i,j\in I}~\prod_{\eta:i\to j}~\map(X_{i},Y_{j})~,
\end{myeq}
\noindent and similarly for \w[.]{\Map_{Z}(X,Y)}

However, \wref{eqtwentyone} is not generally a fibration
sequence, except when the underlying graph of $I$ is a tree 
(the proof of \cite[Prop.~4.23]{BJTurR}, where $I$ consists of a single
map, generalizes to this case). 
Nevertheless, for \good indexing categories $I$ (Definition \ref{dgoodi}),
we can define a subspace \w{\lof{I}{X}{Y}} (see Definition \ref{dlof}) 
inside the right-hand space of \wref[,]{eqtwentyone} such that $\xi$
factors through a fibration $\Psi$ (see Lemma \ref{diffmap} below), and:
\begin{myeq}[\label{eqpsi}]
0\to\Map_{Z}(X,Y)~\to~\trans{Z}{X}{Y}~\xra{\Psi}~\lof{I}{X}{Y}
\end{myeq}
is thus a fibration sequence.

For such an $I$ we obtain an auxiliary tower:
\begin{myeq}[\label{eauxt}]
\lof{I_{n}}{X}{Y}~\xra{p_{n-1}}~\lof{I_{n-1}}{X}{Y}~\to\dotsc\to~
\lof{I_{2}}{X}{Y}~\xra{p_{1}}~\lof{I_{1}}{X}{Y}
\end{myeq}
(see \S \ref{nlof}). We shall show that the maps \w{p_{k}} are fibrations
(see Proposition \ref{ppsifib}), with a fiber which we identify
as \w{F_{k}:=\cH^{I_{k}}_{c}(X/Z,Y)} (cf.\ Definition \ref{dlocoh}).
\end{mysubsection}

\begin{mysubsection}{The auxiliary fibers}
\label{sauxfib}
Since all of these constructions will be natural, for each $k$ the inclusion of categories 
\w{i_{k-1}:I_{k-1}\hra I_{k}} will induce a commuting square of
fibrations:
\[
\xymatrix{
\trans{\CInZ{I_{k}}}{X}{Y} \ar[r]^-{\Psi_{k}} \ar[d]_{\pi_{k-1}} & 
\lof{I_{k}}{X}{Y} \ar[d]^{p_{k-1}} \\
\trans{\CInZ{I_{k-1}}}{X}{Y} \ar[r]^-{\Psi_{k-1}} & 
\lof{I_{k-1}}{X}{Y}~,
}
\]
where the left vertical map \w{\pi_{k-1}} is the projection onto the
appropriate factors. Thus we will have a homotopy-commutative diagram:
\mydiagram[\label{eqsquare}]{
\Fib(i_{k-1}^{\ast}) \ar[r] \ar[d] & 
\prod\limits_{i\in I_{k}\setminus I_{k-1}}\,\Map_{\cC/Z_{i}}(X_{i},Y_{i})
\ar[r] \ar[d] & 
\cH^{I_{k}}_{c}(X/Z,Y) \ar[d] \\
\Map_{\CInZ{I_{k}}}(X,Y) \ar[r] \ar[d]_{i_{k-1}^{\ast}} &
\trans{\CInZ{I_{k}}}{X}{Y} \ar[r]^-{\Psi_{k}} \ar[d]_{\pi_{k-1}} & 
\lof{I_{k}}{X}{Y} \ar[d]^{p_{k-1}} \\
\Map_{\CInZ{I_{k-1}}}(X,Y) \ar[r] & 
\trans{\CInZ{I_{k-1}}}{X}{Y} \ar[r]^-{\Psi_{k-1}} & 
\lof{I_{k-1}}{X}{Y} 
}
in which  all rows and columns are fibration sequences up to homotopy. 

Since the homotopy groups of \w{\Pi_{i}\Map_{\cC/Z_{i}}(X_{i},Y_{i})} are a
direct product of cohomology groups of the individual spaces in the
diagram $X$, the top row of \wref{eqsquare} allows us to identify the
successive homotopy fibers of maps of the primary tower \wref{eqprimary}  
in terms of those of the auxiliary tower \wref[.]{eauxt}
Taking \w[,]{k=n} we see also that \w{\Map_{Z}(X,Y)} is in fact the
homotopy limit of the primary tower.  
\end{mysubsection}

\begin{mysubsection}{A modified primary tower}
\label{smpt}
Using standard methods, we can change \wref{eqprimary} into a tower
with the same homotopy limit, but simpler  successive fibers:

For \w{1 \leq k \leq n}  we define
\w{q_{k}:\trans{Z}{X}{Y}\to\lof{I_{k}}{X}{Y}} to be the
composite fibration: 
$$
\trans{Z}{X}{Y}~\xra{\Psi_{I}}~\lof{I}{X}{Y}~
\xra{p_{k}\circ\dotsc\circ p_{n-1}}~\lof{I_{k}}{X}{Y}~,
$$
and denote the fiber of \w{q_{k}} by
\w[.]{\ctrans{\CInZ{I_{k}}}{I}{X}{Y}} 

The induced maps 
\w{r_{k}:\ctrans{Z}{I_{k}}{X}{Y}\to\ctrans{Z}{I_{k-1}}{X}{Y}} then fit
into a tower:
\begin{myeq}[\label{eqvarp}]
\ctrans{Z}{I_{n}}{X}{Y}~\xra{r_{n-1}}~ \dotsc ~\xra{r_{2}}~ 
\ctrans{Z}{I_{2}}{X}{Y}~\xra{r_{1}}~\ctrans{Z}{I_{1}}{X}{Y}~.
\end{myeq}
As in \S \ref{sauxfib}, we see that the homotopy fiber of \w{r_{k}} is
the loop space of the fiber \w{F_{k}:=\cH^{I_{k}}_{c}(X/Z,Y)} of
\w[,]{p_{k}} while the homotopy limit of \wref{eqvarp} is
\w[.]{\ctrans{Z}{I}{X}{Y}=\Map_{Z}(X,Y)} Therefore, if we take the
homotopy spectral sequence for the tower \wref[,]{eqvarp} rather than
that for \wref[,]{eqprimary} we get the same abutment, and a closely
related \ $E^{2}$-term.
\end{mysubsection}

\begin{defn}\label{dctrans}
For \w{\tuple} as above and $J$ a subcategory of $I$,
we denote by \w{\ctrans{\CIZ}{J}{X}{Y}=\ctrans{Z}{J}{X}{Y}} 
the sub-simplicial set of \w{\trans{Z}{X}{Y}} consisting of
transformations which are natural when restricted to $J$-diagrams. 
In other words, these are elements $\sigma$ of \w{\trans{Z}{X}{Y}}
which make
\mydiagram{
X_{i} \ar[d]_{\sigma_{i}} \ar[r]^{X(f)} & X_{j} \ar[d]^{\sigma_{j}}  \\
Y_{i} \ar[r]_{Y(f)} & Y_{j}
}
commute, for any morphism \w{f:i\to j} in $J$.
\end{defn}

For example, \w[,]{\ctrans{Z}{I_{1}}{X}{Y}} consists of those
transformations which are natural only with respect to morphisms of
maximal length. On the other hand, \w{\ctrans{Z}{I}{X}{Y}} is simply 
\w[.]{\Map_{Z}(X,Y)}

Note that any inclusion of subcategories \w{J' \to J} of $I$ induces
an injection of simplicial sets
\w[,]{r^{J}_{J'}:\ctrans{Z}{J}{X}{Y}\to\ctrans{Z}{J'}{X}{Y}} 
since any transformation natural over $J$ must be natural over the
subcategory \w[.]{J'}

\begin{lemma}\label{lctrans}
For \w{(I_{k})_{k=1}^{n}} as in \wref[,]{eqprimt} we can identify
\w{\ctrans{Z}{I_{k}}{X}{Y}} of \S \ref{smpt} with
\w[,]{\ctrans{\CIZ}{I_{k}}{X}{Y}} and 
\w{r_{k}:\ctrans{Z}{I_{k}}{X}{Y}\to\ctrans{Z}{I_{k-1}}{X}{Y}} with
\w[.]{r^{I_{k}}_{I_{k-1}}}
\end{lemma}

\begin{proof}
Follows from Definition \ref{dctrans}.
\end{proof}

%
%
\section{The Auxiliary Tower}
\label{cauxt}

Suppose \w{\tuple} is admissible. In order to construct the auxiliary
tower \wref[,]{eauxt} we need a number of definitions:

\begin{defn}\label{ddiags}
Assuming \w{\tuple} is admissible:
\begin{enumerate}
\renewcommand{\labelenumi}{\alph{enumi})\ }
\item For any composable sequence \w{\fdot} of $k$ non-identity
  morphisms in $I$ (i.e., a $k$-simplex of the reduced nerve of $I$,
  \w[,]{\cN(I)} where identities are excluded) its \emph{diagonal}
  mapping space is 
$$
\mfd:=\Map_{Z_{t(f_{k})}}(X_{s(f_{1})},Y_{t(f_{k})})~,
$$
In particular, for \w{f:a\to b} in $I$ we have
\w[.]{\mf:=\Map_{Z_{b}}(X_{a},Y_{b})}
\item For each \w[,]{k\geq 1} let 
\w[.]{\displaystyle \diagk{Z}{X}{Y}{k}~:=~
\prod_{\fdot \in\,\cN(I)_{k}}~\mfd}
In particular, we denote \w{\diagk{Z}{X}{Y}{1}=\prod_{f \in I}~\mf}
by\w[.]{\diag{Z}{X}{Y}} 
\item  Any map into the product \w{\diagk{Z}{X}{Y}{k}} is defined
by specifying its projection onto each factor  \w[,]{\mfd} 
indexed by \w[.]{\fdot \in \cN(I)_k}

In particular, we have two maps of interest
\w[:]{\diagk{Z}{X}{Y}{k-1} \to \diagk{Z}{X}{Y}{k}} 
\begin{enumerate}
\renewcommand{\labelenumii}{(\roman{enumii})\ }
\item \w[,]{X^{\ast}}  for which the \ $\fdot$-component is the composite
\[
\diagk{Z}{X}{Y}{k-1}\xra{\proj}
\mdiag{f_{2},\dotsc,f_{k}}
\xra{X(f_{1})^{\ast}} \mfd~.
\]
\item \w[,]{Y_{\ast}}  for which the \ $\fdot$-component is the composite
\[
\diagk{Z}{X}{Y}{k-1}\xra{\proj}
\mdiag{f_{1},\dotsc,f_{k-1}}
\xra{Y(f_{k})_{\ast}} \mfd
\]
\end{enumerate}
\item  By iterating the maps
\w{\Phiof{1}{}:=Y_{\ast}+X^{\ast}:\diagk{Z}{X}{Y}{k-1}\to\diagk{Z}{X}{Y}{k}} 
for various \w{k>1} we obtain maps:
\[
\Phiof{j}{}:\diagk{Z}{X}{Y}{k} \to \diagk{Z}{X}{Y}{k+j}
\]
 for each \w[.]{j \geq 1} Setting
 \w[,]{\Phiof{0}{}:=\Id:\diagk{Z}{X}{Y}{1}\to\diagk{Z}{X}{Y}{1}}
 we may combine these to define:
\[
\Phis{}:\diag{Z}{X}{Y}\to \prod_{k=1}^{n}~\diagk{Z}{X}{Y}{k}~.
\]
For any \w[,]{\fdot\in\cN(I)_{k}} we write \w{\Phis{\fdot}} for $\Phis{}$ 
composed with the projection onto \w[.]{\mfd}
\item For any \w[,]{\fdot=(f_{1},\dotsc,,f_{k})\in\cN(I)_{k}} 
let \w{\comp{\fdot}:=f_{k}\circ f_{k-1}\circ\dotsc\circ f_{1}} 
denote the composition in $I$. We then have  a map 
\w[,]{\kappa_{\fdot}:\prod_{k=1}^{n}\diagk{Z}{X}{Y}{k}\to\mdiag{\comp{\fdot}}}
which is just the projection onto 
\w[.]{\mfd \xra{=} \mdiag{\comp{\fdot}}}
\end{enumerate}
\end{defn}

\begin{remark}\label{rphis}
If \w[,]{(g,f)\in\cN(I)_{2}} is a composable pair in $I$,
then by definition of $\Phi$ we have 
\[
\Phis{(g,f)}=Y(f)\circ\Phis{g}+\Phis{f}\circ X(g)~.
\]
More generally, if \w{\hdot=(\gdot,\fdot)\in\cN(I)_{k+j}} is 
the concatentation of \w{\gdot\in\cN(I)_{k}} and 
\w[,]{\fdot\in\cN(I)_{j}} then:
\begin{myeq}[\label{eqformphi}]
\Phis{(\gdot,\fdot)}~=~Y(\comp{\fdot})_{\ast}\Phis{\gdot}~+~
X(\comp{\gdot})^{\ast}\Phis{\fdot}~.
\end{myeq}

Note also that 
\[
(Y_{\ast}+X^{\ast}) \circ (Y_{\ast}+X^{\ast})~=~
Y_{\ast}Y_{\ast}+Y_{\ast}X^{\ast}+X^{\ast}X^{\ast}:\diagk{Z}{X}{Y}{k}\to 
\diagk{Z}{X}{Y}{k+2}
\]
and so inductively:
\begin{myeq}[\label{eqphi}]
\Phiof{j}{}=(Y_{\ast}+X^{\ast})^j~=~
\Sigma_{i=0}^j (Y_{\ast})^{j-i}(X^{\ast})^i:\diagk{Z}{X}{Y}{k}\to 
\diagk{Z}{X}{Y}{k+j}.
\end{myeq}
\end{remark}

\begin{defn}\label{dlof}
Let \w{K_{I}} denote the indexing category with
\begin{enumerate}
\renewcommand{\labelenumi}{$\bullet$ \ }
\item  objects: \  $\bz$, $\bo$, and \w[,]{\Arr(I):=\cN(I)_{1}} 
\item morphisms: \  one arrow \w[,]{\phi:\bz\to\bo}
and an arrow \w{k_{\fdot}:\bo\to \comp{\fdot}\in\Arr(I)} for each  
\w[.]{\fdot\in\cN(I)} 
\end{enumerate}

If \w{\tuple} is admissible, define a diagram of simplicial abelian
groups \w{V_{I}:K_{I}\to s\cA} by setting \w[,]{V_{I}(\bz)=\diag{Z}{X}{Y}} 
\w[,]{V_{I}(\bo)=\prod_{k=1}^{n}~\diagk{Z}{X}{Y}{k}} and 
\w[,]{V_{I}(f)=\mf} with
\w{V_{I}(\phi)=\Phis{}} and \w[.]{V_{I}(k_{\fdot})=\kappa_{\fdot}}
Then set \ \w[.]{\lof{I}{X}{Y}:=\lim_{K_{I}}~V_{I}}

This limit can be described more concretely as follows:
write \w{\indec{I}} for the collection of indecomposable maps in $I$, 
and let \w{\hlof{I}{X}{Y}} denote the subspace of
\w{\prod_{f \in \indec{I}}~\mf} consisting of tuples \w{\phidot}
satisfying 
\begin{myeq}[\label{eqlof}]
\sum_{i=0}^k Y(f_{k} \circ \dots \circ f_{i+1})\varphi_{f_{i}}X(f_{i-1}
\circ \dots \circ f_{1})~=~
\sum_{i=0}^l Y(g_{l} \circ \dots \circ g_{i+1})\varphi_{g_{i}}X(g_{i-1}
\circ \dots \circ g_{1})
\end{myeq}
whenever \w[.]{\comp{\fdot}=\comp{\gdot}}
\end{defn}

\begin{lemma}\label{llof}
The simplicial abelian group \w{\lof{I}{X}{Y}} is isomorphic to 
\w[.]{\hlof{I}{X}{Y}}
\end{lemma}

\begin{proof}
The limit condition for \w{\varphi \in \lof{I}{X}{Y}} implies that the
value of \w{\varphi_{f}} for any decomposable $f$ is uniquely
determined by the values of \w{\varphi_{f_{i}}} for \w{f_{i}}
indecomposable, by the recursive formula \wref[.]{eqformphi}
\end{proof}

\begin{remark}\label{include}
As a consequence of the previous lemma, for (full)
subcategories \w{J \subset I} we have natural inclusion
maps \w[.]{i_J:\lof{J}{X}{Y} \to\prod_{f \in \indec{J}} \mf}
\end{remark}

We now investigate the properties of \w{\lof{I}{X}{Y}} and its associated
fibrations. First, note that there are two maps 
\w[,]{X^{\ast},Y_{\ast}:\trans{Z}{X}{Y}\to\diag{Z}{X}{Y}} which project to precomposition and postcomposition respectively on appropriate factors and we show:

\begin{lemma}\label{diffmap}
The difference map 
\w{\xi:=Y_{\ast}-X^{\ast}:\trans{Z}{X}{Y}\to\diag{Z}{X}{Y}} 
factors through a map \w{\Psi:\trans{Z}{X}{Y}\to\lof{I}{X}{Y}} 
with kernel \w[.]{\Map_{Z}(X,Y)}
\end{lemma}

\begin{proof}
Note that the sum \wref[,]{eqphi} applied to an element in the image of the 
difference map 
\[
Y_{\ast}-X^{\ast}:\trans{Z}{X}{Y}\to\diag{Z}{X}{Y}~,
\]
is telescopic, so we are left with: \w[.]{(Y_{\ast})^k-(X^{\ast})^{k}}
Since $X$ and $Y$ are in \w[,]{\CI} for any \w{\fdot\in\cN(I)_{k}}
the composite:
\[
\trans{Z}{X}{Y}\to\diag{Z}{X}{Y}\to \prod_{k=1}^n\diagk{Z}{X}{Y}{k} 
~\xra{\kappa_{\fdot}}~\mdiag{\comp{\fdot}}
\]
%
sends any $\sigma$ to \w[.]{Y(f)\sigma_{s(f)}-\sigma_{t(f)}X(f)}
As a consequence, we get an identical value for any \w{\gdot\in\cN(I)_{j}} 
with \w[.]{\comp{\fdot}=\comp{\gdot}}
Thus, the universal property of the limit implies the difference map
factors through the limit \w[.]{\lof{I}{X}{Y}}

To identify the kernel of $\Psi$, we instead consider the difference map:
\[
Y_{\ast}-X^{\ast}:\trans{Z}{X}{Y}\to\diag{Z}{X}{Y}~.
\]
Clearly \w{\Psi(\sigma)=0} if and only if 
\w[,]{Y(f)\sigma_{s(f)}-\sigma_{t(f)}X(f)=0} for every morphism
$f$ in $I$ \ -- \ that is, precisely when $\sigma$ is a natural transformation 
of \w[.]{\CI}  Since both $X$ and $Y$ are diagrams over $Z$, and
each \w{\sigma_{f}} is a map over \w[,]{Z_{f}} $\sigma$  is in that case actually a
natural transformation over $Z$.
\end{proof}

\begin{notation}\label{nlof}
In order to describe the behavior of the $L$-construction with respect 
to the inclusion of a subcategory \w[,]{\iota:J \to I} note that we can 
define two different diagrams of simplicial abelian groups indexed by
\w{K_{J}} (Definition \ref{dlof}):

One is \w[,]{V_{J}} whose limit is \w[.]{\lof{J}{X}{Y}} The second, 
which we denote by \w[,]{V_{I,J}} has \w[,]{V_{I,J}(\bz)=\diag{Z}{X}{Y}} 
\w[,]{V_{I,J}(\bo)=\prod_{k=1}^{n}~\diagk{Z}{X}{Y}{k}} 
as for \w[,]{V_{I}} (and \w{V_{I,J}(f)=\mf} for \w[).]{f\in\Arr(J)} 
If we set \w[,]{\lof{I,J}{X}{Y}:=\lim_{K_{J}}~V_{I,J}}
we see that there is a canonical map
\w{\tau:\lof{I}{X}{Y}\to \lof{I,J}{X}{Y}}
(since fewer constraints are imposed in defining the second limit 
as a subset of \w[).]{\prod_{f \in \indec{I}}~\mf}

On the other hand, we have a morphism of \ $K_{J}$-diagrams from 
\w[,]{\xi:V_{I,J}\to V_{J}} obtained by projecting the 
larger products \w{\diagk{Z}{X}{Y}{k}} onto  
\w{\diagk{Z\rest{J}}{X\rest{J}}{Y\rest{J}}{k}} for each \w[.]{k\geq 1}  
This induces a map on the limits
\w[,]{\xi_{\ast}:\lof{I,J}{X}{Y} \to \lof{J}{X}{Y}} 
and we define the \emph{restriction map}
\w{(p=)p^{I}_{J}:\lof{I}{X}{Y} \to \lof{J}{X}{Y}} to be 
\w[.]{p^I_J:=\xi_{\ast}\circ\tau}

Finally, note that there is an obvious restriction map
\w[,]{r:\trans{\CIZ}{X}{Y}\to\trans{\CInZ{J}}{X}{Y}}
which is simply the projection onto the factors indexed by \w[.]{\Arr(J)}
\end{notation}

From the definitions it is clear that the diagram:
\mydiagram[\label{dilof}]{
\trans{\CIZ}{X}{Y} \ar[r]^-{\Psi_{I}} \ar[d]_{r} & \lof{I}{X}{Y} \ar[d]^{p^I_J} \\
\trans{\CInZ{J}}{X}{Y} \ar[r]_-{\Psi_{J}} & \lof{J}{X}{Y}.
}
commutes.

The kernel of \w{p^I_J\circ\Psi_{I}} will be the same as the
kernel of \w[,]{\Psi_{J} \circ r^I_J} by the commutativity of \wref[.]{dilof}
However, by Lemma \ref{diffmap}, the kernel of \w{\Psi_{J}} is the space
of $J$-natural transformations.  Thus the kernel of the composite 
\w{p^I_J \circ \Psi_{I}} will be the space \w[.]{\trans{\CInZ{J}}{X}{Y}}

%
%

\begin{lemma}\label{ldescript}
Given \w{J\subseteq I} and \w{f \in \indec{J}} with \w{f=\comp{\fdot}} 
for \w{\fdot=(f_{k},f_{k-1},\dotsc,f_{1})\in\cN(I)_{k}} with 
\w{f_{i}\in\indec{I}} ($i=1,\dotsc k$), \ the following diagram commutes:
\mydiagram[\label{didescript}]{
\lof{I}{X}{Y} \ar[r]^{p^I_J} \ar[d]^-{i_{I}} & \lof{J}{X}{Y} \ar[d]^-{i_{J}}\\
{\displaystyle \prod_{f \in \indec{I}}\mf} \ar[d]^{\proj} &
{\displaystyle \prod_{f \in \indec{J}}\mf} \ar[d]^{\proj} \\
\mdiag{f_{1}}\times\dotsc\times\mdiag{f_{k}} \ar[r]^-{\Phiof{k}{\fdot}} & \mf
}
where the maps \w{i_{I}} and \w{i_{J}} are the inclusions of 
\S \ref{include}. 
\end{lemma}

\begin{proof}
Suppose \w{\phidot} is an element of \w[,]{\lof{I}{X}{Y}} while
\w{f=\comp{\fdot}} is a maximal decomposition (so each \w{f_{i}} is
indecomposable). Then \w{\varphi_{f}} lies in \w[,]{\diagk{Z}{X}{Y}{1}} so
\w{\Phis{\varphi_{f}}=\varphi_{f}} lands in \w[.]{\mf} However,
\w{(\varphi_{f_{k}},\dots,\varphi_{f_{1}})\in 
\mdiag{f_{k}}\times\dots\times\mdiag{f_{1}}}
maps to
\w{\Sigma_{i=0}^{k} Y(f_{k} \circ \dots \circ f_{i+1})\varphi_{f_{i}}X(f_{i-1}
\circ \dots \circ f_{1})}
also in \w[.]{\mdiag{\comp{\fdot}}=\mf} Thus, 
\w{\phidot \in \lof{I}{X}{Y}=\hlof{I}{X}{Y}} (see Lemma \ref{llof}) 
implies the value of \w{\varphi_{f}} for any
decomposable $f$ is uniquely determined by the values of
\w{\varphi_{f_{i}}} for \w{f_{i}} indecomposable, using formula
\wref[.]{eqlof}
\end{proof}

Note that if $f$ is also indecomposable in $I$, the bottom 
map of \wref{didescript} is \w[.]{\Id:\mf\to\mf} The choice of 
decomposition of $f$ in $I$ is also irrelevant, by Definition \ref{dlof}.

%
%
\section{Fibrations in the Auxiliary Tower}
\label{cfibat}

As noted in \S \ref{sauxf}, the auxiliary tower \wref{eauxt} was
constructed with two goals in mind: to replace \wref{eqprimary} by a
tower of fibrations (with the same homotopy limit), and to identify the
homotopy fibers of the successive maps in \wref[.]{eqprimary}
In this section we show that the map $\Psi$ of Lemma \ref{diffmap} 
is indeed a fibration, and that the auxiliary tower is a tower of
fibrations. First, we need the following:

\begin{defn}\label{dfiltrs}
Any \good indexing category $I$ has two filtrations, defined inductively:
\begin{enumerate}
\renewcommand{\labelenumi}{\alph{enumi})\ }
\item The filtration \w{\{\cF_{i}\}_{i=0}^{n}} on $I$ is defined by
  decomposition length from the left, so \w{\cF_{0}} consists of
  weakly initial objects in $I$ and
  \w{\cF_{n+1}} consists of indecomposable maps 
  with sources in \w[,]{\cF_{n}} 
  (including their targets).
\item The filtration \w{\{\cG_{i}\}_{i=0}^{n}} is similarly defined by
  decomposition length from the right, so \w{\cG_{0}} consists of the
  weakly terminal objects in $I$ and \w{\cG_{n+1}} consists of indecomposable
  maps with targets in \w[,]{\cG_{n}} (including their sources).
\end{enumerate}
\end{defn}

\begin{prop}\label{ppsifib}
If \w{\tuple} is admissible, the induced difference map:
\[
\Psi:\trans{Z}{X}{Y}\to \lof{I}{X}{Y}
\]
of Lemma \ref{diffmap} is a fibration of simplicial abelian groups.
\end{prop}

\begin{proof}
By \cite[II, \S 3, Prop.~1]{QuiH},  it suffices to show that $\Psi$
surjects onto the zero component of \w[.]{\lof{I}{X}{Y}} 
Thus, given \w[,]{0\sim\phidot\in\lof{I}{X}{Y}} we must produce an 
element \w{\sigdot \in\trans{Z}{X}{Y}} with \w[;]{\Psi(\sigdot)=\phidot} 
i.e., for every \w{f:a\to b} in $I$ we want:
\begin{myeq}[\label{eqsigma}]
\sigma_{b}\circ X(f)~=~Y(f)\circ\sigma_{a}-\phiof{f}~.
\end{myeq}

Note that since $Y$ is an abelian group object in \w[,]{\CIZ}
the zero map \w{X\to Y} is the unique map in \w{\CIZ} that
factors through the section \w{s:Z \to Y} (which exists by
\wref{eeight} and \S \ref{ssketch}). 

We proceed by induction on the filtration \w{\{\cF_{i}\}_{i=0}^{n}} of $I$
of Definition \ref{dfiltrs}. To begin, for each \w[,]{c\in\cF_{0}} we 
may choose \w{\sigma_{c}:X_{c}\to Y_{c}} to be $0$.

Assume by induction that we have constructed maps
\w{\sigma_{c}:X_{c}\to Y_{c}} for each \w[,]{c\in\cF_{i}}
satisfying \wref{eqsigma} for every $f$ in \w[,]{\cF_{i}}
and with each \w[.]{\sigma_{c}\sim 0}
Note that for any \w[,]{f:b\to c} in \w{\cF_{i+1}} the map:
\begin{myeq}[\label{eqnusigma}]
\nuof{f}~:=~Y(f)\circ\sigma_{b}-\phiof{f}:X_{b}\to Y_{c}
\end{myeq}
is well-defined (since necessarily \w[).]{b\in\cF_{i}} This is our
candidate for \w[,]{\sigma_{c}\circ X(f)} and
\w{\nuof{f} \sim -Y(f)\circ\sigma_{b} \sim 0} by the assumption on
$\varphi$ together with the induction hypothesis (considering naturality
of the section
\w[).]{Z \to Y}  

Moreover, given any \w{g:a\to b} (necessarily in \w[),]{\cF_{i}} we
have \w{\phiof{g}=Y(g)\circ\sigma_{a}+\sigma_{b}\circ X(f)} by 
\wref[,]{eqsigma} so from \w{\phidot \in \lof{J}{X}{Y}} it follows that:
\begin{myeq}[\label{eqcompnu}]
\begin{split}
\nuof{f \circ g}~=&~Y(f\circ g)\circ\sigma_{a}-\phiof{f\circ g}\\
=&~Y(f\circ g)\circ \sigma_{a}-[Y(f)\circ\phiof{g}+\phiof{f}\circ X(g)]\\
=&~Y(f\circ g)\circ\sigma_{a}-
[Y(f) \circ (Y(g)\circ\sigma_{a}-\sigma_{b}\circ X(g))+\phiof{f}\circ X(g)] \\
=&~\nuof{f}\circ X(g) .
\end{split}
\end{myeq}

Now given \w[,]{c\in\cF_{i+1}\setminus\cF_{i}} set:
\[
\hX_{c}~:=~\colimit{b\in I/c}X_{b}~.
\]

Since \w{X\in\CI} is cofibrant, it is Reedy cofibrant (\S
\ref{rmodel}), which implies that the canonical map
\w{\varepsilon_{c}:\hX_{c}\to X_{c}} is a cofibration. Moreover,
\wref{eqcompnu} implies  that the maps \w{\nuof{f}} defined above induce a map
\w[.]{\hat{\nu_c}:\hX_{c}\to Y_{c}} Since all the maps in question 
are nullhomotopic by construction, the diagram:
\[
\xymatrix{
\hX_{c} \ar[rd]_{\hat{\nu_c}}\ar[r]^{\varepsilon_{c}} & X_{c}\ar[d]^{0} \\
& Y_{c}\\
}
\]
commutes up to homotopy. Hence by \cite[Cor.~4.20]{BJTurR}  there is a map
\w{\sigma:X_{c}\to Y_{c}} in \w{\cC/Z_{c}} making the diagram
\mydiagram[\label{diextend}]{
\hX_{c} \ar[rd]_{\hat{\nu_c}}\ar[r]^{\varepsilon_{c}} & X_{c}\ar[d]^{\sigma} \\
& Y_{c}\\
}
commute, and we choose this to be \w[.]{\sigma_{c}} By construction
\w{\sigma_{c}\circ X(f)=\nuof{f}} for every \w[,]{f:b\to c} so 
\wref{eqsigma} is satisfied. This completes the induction.
\end{proof}

\begin{prop}\label{pfibl}
If \w{\tuple} is admissible, let $J$ be a subcategory of $I$
obtained by omitting a terminal object $c$. Then the restriction map
\w{p^I_J:\lof{I}{X}{Y}\to\lof{J}{X}{Y}} is a fibration.  
\end{prop}

\begin{proof}
As in the previous proof, we must inductively define a lift 
\w{\sigdot\in \lof{I}{X}{Y}} 
for a nullhomotopic \w[.]{\phidot \in \lof{J}{X}{Y}} 
Under these conditions, \w{p^I_J} is simply a forgetful functor, so 
this means \w{\sigma_g=\varphi_g} for $g$ a morphism of $J$ and
we must define \w{\sigma_{\ell}:X_d \to Y_c} whenever \w{\ell:d \to c}
is a morphism in $I$, in a manner compatible with the definition of
\w[.]{\phidot} Note that \w{\phidot} determines the composite
\w[.]{Y(f) \circ \Phiof{I}{\gdot}=:\psiof{\gdot,f}}

Following the approach of the previous proof, we will define
\w{\nuof{\gdot,f}}  for any \w{e\xra{\gdot} d\xra{f} c} in $I$, where $f$ 
is indecomposable, so as to satisfy three properties:

First, we require that our choices be \emph{coherent}:
\begin{myeq}[\label{eqmucohere}]
\nuof{\gdot \circ \hdot,f}=\nuof{\gdot,f} \circ X(\comp{\hdot})~,
\end{myeq}
which will allow us to build a homotopy commutative triangle using a
colimit construction.  

Second, we need our choices to be \emph{consistent}:
\begin{myeq}[\label{eqmuconsist}]
\nuof{\gdot,f}=\nuof{\gdot',f'} + \psiof{\gdot,f} - \psiof{\gdot',f'} 
 \text{ whenever } f \circ \gdot=f' \circ \gdot' \text{ in } I~, 
\end{myeq}
which is needed so that we eventually obtain an element 
\w[.]{\sigdot \in \lof{I}{X}{Y}} In fact, our construction will also
work when \w[,]{\gdot=\emptyset} which will yield 
\w[.]{\sigof{f}=\nuof{\emptyset,f}}

Finally, we require that each \vsm \w[.]{\nuof{\gdot,f} \sim 0}

We now proceed to choose \w{\nuof{\gdot,f}} for \w{e\xra{\gdot}d\xra{f} c}
with \w{e\in\cF_{i}} (Definition \ref{dfiltrs}) by induction on \w[:]{i\geq 0}

For each \w{\ell:e\to c} in $I$ with \w[,]{e\in\cF_{0}} choose some 
decomposition \w{e\xra{\gdot} d\xra{f} c} (with \w{\ell=\comp{\gdot,f}} 
and \ $f$ indecomposable), and an arbitrary nullhomotopic
\w[.]{0=\nuof{\gdot,f}:X_{e}\to Y_{c}} For any other decomposition
\w[,]{\ell=\comp{\gdot',f'}} the map \w{\nuof{\gdot',f'}} is then
determined by \wref[.]{eqmuconsist}

Assume that $\nu$ has been defined for every \w{e \in \cF_{i}} 
so that \wref{eqmucohere}
and \wref{eqmuconsist} hold (wherever applicable). \ For each 
\w{e\in\cF_{i+1}\setminus\cF_{i}} and map \w[,]{\ell:e\to c} consider the
over-category \w{\cF_{i}/e} (which is non-empty by definition of
\w[)]{\cF_{i+1}} and set \w[.]{\hX_{e}:=\colim_{a\in\cF_{i}/e}\,X_{a}}
Because the diagram $X$ is cofibrant, hence Reedy cofibrant (\S \ref{rmodel}) 
in \w[,]{\CI} the canonical
map \w{\varepsilon_{e}:\hX_{e}\hra X_{e}} is a cofibration.

Again choose some decomposition \w{e\xra{\gdot} d\xra{f}c} of $\ell$.
The maps \w[,]{\nuof{\gdot \circ \hdot,f}:X_{a}\to Y_{c}} for each composable
sequence \w{\hdot:a\to e} in \w{\cF_{i}/e} induce a (necessarily
nullhomotopic) map \w{\hat{\mu}_{e}:\hX_{e}\to Y_{c}} 
by \wref[.]{eqmucohere}
Since:
\[
\xymatrix{
\hX_{e} \ar[r]^{\varepsilon_{e}} \ar[rd]_{\hat{\nu}_{(\gdot,f)}} & X_{e}
\ar[d]^{0}\\
& Y_{c}
}
\]
then commutes up to homotopy, we apply \cite[Cor.~4.20]{BJTurR}
to find
\[
\xymatrix{
\hX_{e} \ar[r]^{\varepsilon_{e}} \ar[rd]_{\hat{\nu}_{(\gdot,f)}} & X_{e}
\ar[d]^{\nuof{\gdot,f}}\\
& Y_{c}
}
\]
making the diagram commute.  

For any other decomposition \w{e\xra{\gdot'} d'\xra{f'} c} of $\ell$,
use \wref{eqmuconsist} to define \w[.]{\nuof{\gdot',f'}}
This completes the induction step.

We have thus defined \w{\nuof{\gdot,f}:X_{e}\to Y_{c}} satisfying 
\wref{eqmucohere} and  \wref{eqmuconsist} for every 
\w{e\xra{\gdot} d\xra{f} c} in \w[.]{I/c} In particular, we can choose 
\w{\sigof{f}=\nuof{\emptyset,f}:X_{d}\to Y_{c}} for each
indecomposable \w{f:d\to c} in $I$ and see that \w{\sigdot\in\lof{I}{X}{Y}} 
(by Lemma \ref{llof}) is the desired lift.
\end{proof}

\begin{cor}\label{cfibl}
If \w{\tuple} is admissible, let $J$ be a full subcategory of
$I$ obtained by omitting an object $c$ such that all maps out of $c$
are indecomposable. Then \w{p^I_J:\lof{I}{X}{Y}\to\lof{J}{X}{Y}} is a fibration. 
\end{cor}

\begin{proof}
As in the proof of Proposition \ref{pfibl} we can construct 
\w{\sigof{f}} for each \w[.]{f:d\to c} in $I$, such that we have 
\w[,]{\hat{\nu}:\colim_{d\in I/c}\,X_{d}\to Y_{c}}
as well as \w[.]{\hat{\epsilon_c}:\colim_{d\in I/c}\,X_{d}\to X_{c}}
For any \w[,]{g:c\to b} in $I$ (indecomposable by assumption),
we also have a map 
\w{\hat{\varphi}:\colim_{d\in I/c}\,X_{d}\to X_{b}}
induced by \w[.]{\phidot} 
Note that by \wref{eqformphi} we must have: 
\[
\sigof{g}\circ X(\hat{\epsilon_c})~=~\Phiof{I}{(g,f)}-Y(g)\circ\sigof{f}~=~
\hat{\varphi}-Y(g)\circ\hat{\nu}~,
\]
and since \w{X(f)} is a cofibration, we may choose the extension 
\w{\sigof{g}} as in \wref[.]{diextend}
\end{proof}

\begin{defn}\label{dstrongf}
If $I$ is a \good indexing category, let \w{\cJ=\{J_{k}\}_{k\in N}}
be a fine orderable cover (\S \ref{emainex}) of $I$ subordinate to the
filtration $\cG$ (Definition \ref{dfiltrs}), such that 
\w{J_{k}\setminus J_{k-1}} consists of a single object of $I$ for each
\w[.]{k\in N} 
Let \w{\cC=s\cA} for some $\fG$-sketchable category $\cA$ (\S
\ref{ssketch}), with \w{Z\in\CI} fibrant. A fibrant abelian group object 
\w{Y\in\CIZ} is called \emph{strongly fibrant} if it is 
$\cJ$-fibrant with respect to the model category structure
of \S \ref{smodel}(a).
\end{defn}

\begin{remark}\label{rstrongf}
Note that this definition is independent of the choice of the
refinement $\cJ$ of $\cG$. Forthermore, by Proposition \ref{ptwo}, any
abelian group object \w{Y\in\CIZ} is weakly equivalent to one which is
strongly fibrant.
\end{remark}

\begin{prop}\label{pfibr}
Suppose \w{\tuple} is admissible, and that $Y$ is strongly fibrant.
Assume that $J$ is obtained from $I$ by 
omitting an object $c$ such that all maps into $c$ are indecomposable.
Then the restriction map
\w{p^I_J:\lof{I}{X}{Y}\to\lof{J}{X}{Y}} is a fibration. 
\end{prop}

\begin{proof}
Dual to the proofs of Proposition \ref{pfibl}   and Corollary \ref{cfibl}.
The strong fibrancy is needed since in the model category we use for 
diagrams ordinary fibrancy is merely objectwise, while strong fibrancy
is dual to Reedy cofibrancy for our purposes.
\end{proof}

\begin{prop}\label{pfibrest}
If \w{\tuple} is admissible, $Y$ is strongly fibrant, and $J$ is
obtained from $I$ by omitting any object $c$, then the restriction map
\w{p^I_J:\lof{I}{X}{Y}\to\lof{J}{X}{Y}} is a fibration.  
\end{prop}

\begin{proof}
Consider any composable sequence:
\begin{myeq}[\label{eqmaparoundc}]
d\xra{\hdot} c\xra{g} b\xra{\fdot} a
\end{myeq}
in $I$. 
As above, \w{0 \sim \phidot \in \lof{J}{X}{Y}}
will determine the map
\begin{myeq}[\label{eqrhoof}]
\psiof{\hdot,g,\fdot}~:=~Y(\comp{(g,\fdot)})\circ\Phiof{I}{\hdot}
~+~\Phiof{I}{\fdot}\circ X(\comp{(\hdot,g)})
\end{myeq}
and we use \w[,]{\nuof{\hdot,g,\fdot}:X_{d}\to Y_{a}} to denote the candidate
for \w{Y(\comp{\fdot})\circ\sigof{g}\circ X(\comp{\hdot})} which we will construct.

As before we require \emph{coherence}:
\begin{myeq}[\label{eqnucohere}]
\nuof{\hdot \circ \ldot,g,\kdot \circ \fdot}~=~
Y(\comp{\kdot})\circ\nuof{\hdot,g,\fdot}\circ X(\comp{\ldot})
\end{myeq}
for any
\[
e\xra{\ldot} d\xra{\hdot} c\xra{g} b\xra{\fdot} a\xra{\kdot} z
\]
in $I$; and \emph{consistency}:
\begin{myeq}[\label{eqnuconsist}]
\nuof{\hdot',g',\fdot'}~=~\psiof{\hdot,g,\fdot}+\nuof{\hdot,g,\fdot}-
\psiof{\hdot',g',\fdot'}
\end{myeq}
whenever \w[.]{\comp{\hdot',g',\fdot'}=\comp{\hdot,g,\fdot}}
\\

We choose the maps $\nu$ satisfying \wref{eqnucohere} and 
\wref{eqnuconsist} by two successive inductions: 
\begin{enumerate}
\renewcommand{\labelenumi}{$\bullet$ \ }
\item The first is by induction on $i$, the filtration degree of $d$ in 
\w{\{\cF_{i}\}_{i=0}^{m}} (by composition length from the left): this
is done as in the proof of Proposition \ref{pfibl}, until finally we have 
\w{\nuof{h,g,\fdot}} for every \w[,]{d\xra{h} c\xra{g} b\xra{\fdot}a} 
where $h$ is indecomposable and $a$ is terminal in $I$ (by coherence
this extends back to any \w[).]{d\xra{\hdot} c} 
\item The second is by induction on $j$, the filtration degree of $a$ in 
\w{\{\cG_{j}\}_{j=0}^{n}} (by composition length from the right), as
in the proof of Proposition \ref{pfibr} (which is why we need $Y$ to be 
strongly fibrant).
\end{enumerate}

At the end of the two induction processes we have chosen
\w{\nuof{h,g}:X_{d}\to Y_{b}} for $h$ and $g$ indecomposable. We can
then choose \w{\sigof{h}=\nuof{h}:X_{d}\to Y_{c}} as in the last step
of the proof of Proposition \ref{pfibl}, and finally choose 
\w{\sigof{g}=\nuof{g}:X_{c}\to Y_{b}} as in the proof of Corollary \ref{cfibl}.
This completes the construction of a lift \w{\sigdot \in \lof{I}{X}{Y}} for 
\w{\phidot} as required.
\end{proof}

\begin{cor}\label{cfib}
Suppose \w{\tuple} is admissible, $Y$ is strongly fibrant, and $J$ is any full
subcategory of $I$ with the same weakly initial and final objects. 
Then the restriction map \w{p:\lof{I}{X}{Y}\to\lof{J}{X}{Y}} is a fibration
\end{cor}

\begin{proof}
By induction on the number of objects in \w[,]{I\setminus J}
using Proposition \ref{pfibrest}.
\end{proof}

%
%
\section{Identifying the Fibers}
\label{cif}

As we have just seen, if $I$ is a good indexing category,
under our standard assumptions on $Z$, $X$, and $Y$ the auxiliary tower 
\wref{eauxt} is a tower of fibrations of simplicial abelian groups. 
It remains to identify the fibers of the restriction maps 
\w[,]{p:\lof{I}{X}{Y}\to\lof{J}{X}{Y}} for a subcategory $J$ of $I$;
this will allow us to determine those of the primary tower
\wref{eqprimary} (or, more directly, those of the modified tower 
\wref[).]{eqvarp} We consider only the case when
\w{I\setminus J} consists of a single internal object $c$.  

\begin{lemma}\label{lfiber}
If \w{\tuple} is admissible and $Y$ is strongly fibrant, then
\w{\phidot\in\Ker(p)\subseteq\lof{I}{X}{Y}} if and only if
\begin{enumerate}
\renewcommand{\labelenumi}{\alph{enumi})\ }
\item \w{\phiof{f}=0} for each morphism $f$ of $I$ which does not
  begin or end in $c$.
\item for any \w{d \xra{g} c \xra{f} b} in $I$ with $f$ and $g$
  indecomposable:
\begin{myeq}[\label{eqkerp}]
Y(f)\circ\phiof{g}+\phiof{f}\circ X(g)=0~,
\end{myeq}
\end{enumerate}
\end{lemma}

\begin{proof}
This follows from Lemma \ref{ldescript}.
\end{proof}

\begin{remark}\label{rsign}
The lemma  implies that \w{(\phiof{f},-\phiof{g})} defines a map from 
\w{X(g)} to \w[.]{Y(f)}
Note also that if \w{\varphi_{f}} is an arrow over \w[,]{Z_{t(f)}} 
the same is true of its negative; the remainder of the diagram for a
map over \w{Z(f)} already commutes because $X$ and $Y$ are diagrams
over $Z$. Thus \w{(\phiof{f},-\phiof{g})} is a map of arrows 
over \w[.]{Z(f)}
\end{remark}

\begin{defn}\label{dlocoh}
If \w{\tuple} is admissible, we define the \emph{local cohomology} of
\w{X\in\CIZ} at an object \w[,]{c\in I} denoted by
\w[,]{\cH_{c}(X/Z,Y)} to be the total derived functors into simplicial
abelian groups of \w{\Map_{\phi_{c}}(\psi_{c},\rho_{c})} applied to
$X$, where 
\w[,]{\psi_{c}:\hocolim\limits_{d\in I/c}\,X_{d}\to X_{c}}
\w[,]{\rho_{c}:Y_{c}\to\holim\limits_{b\in c/I}\,Y_{b}} and
\w[,]{\phi_{c}:Z_{c}\to\holim\limits_{b\in c/I}\,Z_{b}} are the
structure maps. The $i$-th \emph{local cohomology group} of \w{X\in\CIZ} 
at $c$ is defined to be \w[.]{\cH^{i}_{c}(X/Z,Y):=\pi_{i}\cH_{c}(X/Z,Y)}
\end{defn}

\begin{remark}
In many cases, the local cohomology at $c$ can be identified
explicitly as the Andr\'{e}-Quillen cohomology of an appropriate
(small) diagram.
\end{remark}

\begin{prop}\label{pfiber}
If \w{\tuple} is admissible, $Y$ is strongly fibrant, and \w[,]{J=I\setminus\{c\}}
then \w{Ker(p)} is weakly equivalent (as a simplicial abelian group)
to \w[.]{\cH_{c}(X/Z,Y)} 
\end{prop}

\begin{proof}
To obtain the total derived functors, in this case, we must replace
$X$ by a weakly equivalent cofibrant, hence Reedy cofibrant object,
which implies that \w{\hocolim\limits_{d\in I/c}\,X_{d}} is simply the
colimit, and \w{\psi_{c}} is a cofibration. By Remark \ref{rstrongf}, we
can replace $Y$ by a weakly equivalent strongly fibrant abelian group
object in \w[,]{\CIZ} which implies that \w{\holim\limits_{b\in c/I}\,Y_{b}}  
is the limit, and \w{\rho_{c}} is a fibration. With these choices, 
\w{\cH^{I}_{c}(X/Z,Y)} is simply the mapping space
\w[,]{\Map_{\phi_{c}}(\psi_{c},\rho_{c})} which is isomorphic to
\w{Ker(p)} in Lemma \ref{lfiber} (using the sign of Remark \ref{rsign}). 
\end{proof}

\begin{thm}\label{tthree}
If \w{\tuple} is admissible, for any ordering 
\w{(c_{i})_{i=1}^{\infty}} of the objects of $I$,
there is a natural first quadrant spectral sequence with: 
\[
E^{2}_{s,t}~=~\cH^{s+1}_{c_{t}}(X/Z;Y)~\Longrightarrow H^{s+t+1}(X/Z;\,Y)~,
\]
with \w[.]{d_{2}:E^2_{s,t} \to E^2_{s-2,t+1}} 
\end{thm}

\begin{proof}
We may replace $Y$ by a weakly equivalent strongly fibrant abelian
group object, by Remark \ref{rstrongf}. By Corollary \ref{cfib},
\wref{eqvarp} is then a tower of fibrations, so it has an 
associated homotopy spectral sequence.  To identify the \ $E^{2}$-term,
note that the homotopy groups of the homotopy fibers of the tower are 
the local cohomology groups in Proposition \ref{pfiber}, suitably
indexed (see Remark \ref{rindex}).
\end{proof}

\begin{remark}\label{rtthree}
Note that \w{p^{I}_{J}:\lof{I}{X}{Y}\to\lof{J}{X}{Y}} is a fibration for 
any full subcategory \w{J\subseteq I} with the same weakly initial and
final objects (Corollary \ref{cfib}), and we
can similarly describe the fiber of \w{p^{I}_{J}} as a sort of 
local cohomology \w[,]{\cH^{I}_{J}(X/Z,Y)} and thus identify the \ 
$E^{2}$-term of the spectral sequence obtained from a fairly arbitrary
cover of $I$.

We shall not attempt to define \w{\cH^{I}_{J}(X/Z,Y)} in
general. Observe, however, that if $J$ is discrete (i.e., there are
no non-identity maps between its objects \w[),]{c_{1},\dotsc, c_{n}} then 
\begin{myeq}[\label{eqhprod}]
\cH^{I}_{J}(X/Z,Y)~\cong~\prod_{i=1}^{n}\ \cH_{c_{i}}(X/Z,Y) .
\end{myeq}
\end{remark}

\begin{example}\label{egsquartwo}
For the commuting square of Example \ref{egsquare}, we now get
a cover for $I$ consisting of \w[,]{I_{3}=I} 
\w{I_{2}=I\setminus\{3\}} \ -- \ i.e., a commuting triangle:
\[
\xymatrix{
4 \ar[d]_{c} \ar[rd]^{b\circ d}  \\
2 \ar[r]_{a} & d
}
\]
\w[,]{I_{1}~=~\{4\xra{a \circ c} 1\}}  and \w[.]{I_0=\{ 4\}}  

Given a diagram of abelian group objects \w[,]{Y:I\to\cC} the local
cohomology groups which form the \ $E^{2}$-term
of the spectral sequence of Theorem \ref{tthree} are:
\[
E^{2}_{s,t}~\cong~\begin{cases} 
H^{s+3}(X(d);\,Y(b))               & \text{if} \ t=2;\\
H^{s+2}(X(c);\,Y(a))               & \text{if} \ t=1;\\
H^{s+1}(X_{4};\,Y_{1})              & \text{if} \ t=0;\\
0 & \text{otherwise.}
\end{cases}
\]

Once more we could unite the first and second rows by omitting 
\w{I_{2}} from our cover, as in Example \ref{egsquare}, by \wref[.]{eqhprod}
\end{example}

\begin{mysubsection}{A comparison}
\label{scompar}
In the simplest case, when \w{I=[\bo]} (a single map):
$$
\xymatrix@R=25pt{
& X_{2} \ar[d]^{f_{2}} \ar[ldd]_{p_{2}} \ar[r]^{X\phi}  &
  X_{1} \ar[d]_{f_{1}} \ar[rdd]^{p_{1}} & \\
& Y_{2} \ar[ld]^<<<<{q_{2}} \ar[r]_{Y\phi}  &
  Y_{1} \ar[rd]_<<<<{q_{1}} & \\
  Z_{2} \ar[rrr]^{Z\phi}  & & &  Z_{1}~,
}
$$
we have the ``defining fibration sequence'':
\begin{myeq}[\label{eqdefin}]
\map(X,Y)~\to~
\map(X_{2},Y_{2})\times\map(X_{1},Y_{1})\xra{\xi}~\map(X_{2},Y_{1})
\end{myeq}
\noindent of \cite[Prop.\ 4.20]{BJTurR} (where all mapping spaces are
taken in the appropriate comma categories).

Projecting the total space of \wref{eqdefin} onto the second
factor yields the following interlocking diagram of horizontal and
vertical fibration sequences:
%
\mydiagram[\label{esix}]{
\map(X_{2},\Fib(Y\phi))\ar[d]_{i_{\ast}} \ar[r] &
\map(X,Y) \ar[r] \ar[d] & \map(X_{1},Y_{1}) \ar[d]^{\Id}\\
\map(X_{2},Y_{2}) \ar[d]_{\phi_{\ast}}\ar[r] &
\map(X_{2},Y_{2})\times\map(X_{1},Y_{1})\ar[r]^<<<<<{\pi} \ar[d]^{\xi} &
\map(X_{1},Y_{1}) \ar[d]\\
\map(X_{2},Y_{1}) \ar[r]^{\Id} & \map(X_{2},Y_{1})\ar[r]& \ast
}

We see that the spectral sequence of Theorem \ref{tone} reduces to the long
exact sequence in homotopy for the top horizontal fibration sequence
in \wref{esix}, while the long exact sequence of Fact \ref{frelc}
is obtained from the left vertical fibration sequence in \wref[.]{esix}
\end{mysubsection}

\begin{remark}\label{rinter}
This actually works for any linear order \w{I=[\bn]} (\S \ref{emainex}):

Given \w[,]{X,Y\in\CIZ} if we set \w{I':=[\bnm]} (so
\w[)]{J:=\{n~\xra{\phi_{n}}~n-1\}} and let
\w[,]{\tau=\tau^{I}_{I'}:\CIZ\to\CInZ{I'}\rest{I'}} then \wref{eqdefin} 
yields a fibration sequence:
$$
\map(X,Y)~\to~
\map(X_{n},Y_{n})\times\map(\tau X,\tau Y)\xra{\xi}~\map(X_{n},Y_{n-1})
$$
which again induces a interlocking diagram of fibrations:
$$
\xymatrix@R=25pt{
\map(X_{n},\Fib(Y\phi_{n}))\ar[d]_{i_{\ast}} \ar[r] &
\map(X,Y) \ar[r] \ar[d] & \map(\tau X,\tau Y) \ar[d]^{\Id}\\
\map(X_{n},Y_{n}) \ar[d]_{(\phi_{n})_{\ast}}\ar[r] &
\map(X_{n},Y_{n})\times\map(\tau X,\tau Y)\ar[r]^<<<<<{\pi} \ar[d]^{\xi} &
\map(\tau X,\tau Y) \ar[d]\\
\map(X_{n},Y_{n-1}) \ar[r]^{\Id} & \map(X_{n},Y_{n-1})\ar[r]& \ast
}
$$
\noindent as in \wref[.]{esix} Note that the long exact sequences in
homotopy (i.e., cohomology) of the central vertical fibrations (for
various values of $n$) provide an alternative inductive approach to
calculating the cohomology of $X$, which can again be formalized in a
spectral sequence (though in this case the fibers are the unknown 
quantity).
\end{remark}

%
%

%
\end{document}